\pdfoutput=1
\documentclass{article}
\usepackage[utf8]{inputenc}
\usepackage{amsmath}
\usepackage{amsthm}
\usepackage{amssymb}
\usepackage{graphicx}
\usepackage{subcaption}
\usepackage{enumerate}
\usepackage{algorithm}
\usepackage{algpseudocode}

\theoremstyle{plain}
\newtheorem{THM}{Theorem}
\newtheorem*{THME}{Theorem 18}
\newtheorem{LEM}[THM]{Lemma}
\newtheorem{PROP}[THM]{Proposition}
\newtheorem{COR}[THM]{Corollary}
\newtheorem{FACT}[THM]{Fact}

\theoremstyle{definition}
\newtheorem*{DEF}{Definition}
\newtheorem*{RMK}{Remark}
\newtheorem{EXA}{Example}
\newtheorem*{ACK}{Acknowledgments}

\DeclareMathOperator{\Mod}{Mod}
\DeclareMathOperator{\SL}{SL}
\DeclareMathOperator{\PSL}{PSL}
\DeclareMathOperator{\GL}{GL}
\DeclareMathOperator{\tr}{tr}
\DeclareMathOperator{\VERT}{Vert}

\newcommand{\hp}{\mathbb{H}^2}
\newcommand{\infcup}[1][k]{\bigcup_{#1=-\infty}^\infty}
\newcommand*\Let[2]{\State #1 $\gets$ #2}

\title{On Translation Lengths of Anosov Maps on Curve Graph of Torus}
\author{\textsc{Hyungryul Baik, Changsub Kim}\\
 \textsc{Sanghoon Kwak, Hyunshik Shin}}

\begin{document}

\maketitle
\begin{abstract}
  We show that an Anosov map has a geodesic axis on the curve graph of
  a torus. The direct corollary of our result is the stable
  translation length of an Anosov map on the curve graph is always a
  positive integer. As the proof is constructive, we also provide an
  algorithm to calculate the exact translation length for any given
  Anosov map. The application of our result is threefold: (a) to
  determine which word realizes the minimal translation length on the
  curve graph within a specific class of words, (b) to establish the
  effective bound on the ratio of translation lengths of an Anosov map
  on the curve graph to that on Teichm{\" u}ller space, and (c) to
  estimate the overall growth of the number of Anosov maps which have
  a sufficient number of Anosov maps with the same translation length.
\end{abstract}

\section{Introduction}
Let $S=S_{g,n}$ be an orientable surface with genus $g$ and $n$
punctures. The \textbf{mapping class group} of $S$, denoted by
$\mathrm{Mod}(S)$, is the group of isotopy classes of
orientation-preserving homeomorphisms of $S$. An element of a mapping
class group is called a \textbf{mapping class}.  The Nielsen--Thurston
classification theorem \cite{thurston1988geometry} states that every
mapping class is either periodic, reducible, or pseudo-Anosov. A
mapping class $\psi$ in $\Mod(S)$ is said to be \textbf{pseudo-Anosov}
if there exists a pair of transverse invariant singular measured
foliations associated to $\psi$, one of which is stretched by a
constant $\lambda$, while the other is contracted by $\lambda^{-1}$.
When $S$ is the torus $T^2$, there are foliations without
singularities and in such case, $\psi$ is called an \textbf{Anosov}
element. It is an easy consequence that Anosov elements can be
represented as matrices $M \in \SL(2,\mathbb{Z}) \cong \Mod(T^2)$ with
$|\tr(M)| > 2$.

The \textbf{curve complex} $\mathcal{C}(S)$ of $S$, first introduced
by Harvey \cite{harvey1981boundary}, is a simplicial complex where
vertices are isotopy classes of essential simple closed curves and
$(k+1)$-vertices span a $k$-simplex if and only if each $(k+1)$-tuple
of vertices has a set of representative curves with the minimal
possible geometric intersection in the given surface. The
\textbf{curve graph} is the 1-skeleton of the curve complex. By giving
each edge length 1, $\mathcal{C}(S)$ becomes a Gromov-hyperbolic metric space
with path metric $d_{\mathcal{C}}(\cdot,\cdot)$
\cite{masur1999geometry}. Then $\Mod(S)$ acts on $\mathcal{C}(S)$ by
isometry. The \textbf{stable translation length}(also known as
asymptotic translation length) of a mapping class $f\in \Mod(S)$ on
$\mathcal{C}(S)$ is defined by
  \[
   l_{\mathcal{C}}(f)=\liminf_{j\rightarrow \infty}\frac{d_{\mathcal{C}}(\alpha,f^j(\alpha))}{j},
  \]
  where $\alpha$ is a vertex in $\mathcal{C}(S)$.  Using the
  triangular inequality, one can show that the $l_{\mathcal{C}}(f)$ is
  independent of the choice of $\alpha$. Masur and Minsky
  \cite{masur1999geometry} showed that $l_{\mathcal{C}}(f) > 0$ if and
  only if $f$ is pseudo-Anosov.

There have been many research works on estimating stable translation length for
non-sporadic surfaces, that is, the complexity $\xi(S) = 3g-4+n$ is positive.
Those works can be found in \cite{farb2008lower},
\cite{Gadre2011minimal}, \cite{Gadre2013Lipschitz},
\cite{valdiva2014asymptotic}, \cite{aougab2015pseudo},
\cite{kin2017small}, \cite{baik2018minimal} and references
therein.
Algorithmic approaches to calculating stable translation lengths of pseudo-Anosov maps
are available in \cite{Shackleton2012} and \cite{webb2015combinatorics}.
A polynomial-time algorithm is established by \cite{bell2016polynomial}.

As far as we know, there is no literature developing a similar theory
for sporadic surfaces. As the complexity of sporadic surfaces is low enough,
it is to be expected that \textit{exact} translation lengths can be
\textit{calculated}, rather than merely estimating stable translation
lengths. Among sporadic cases, only $S=S_{0,4},S_{1,0},S_{1,1}$ have the
\textbf{Farey graph} $\mathcal{F}$ as their curve graph; the other ones, namely the
spheres with at most 3 punctures have the empty set as their curve graph.
Thus, in this paper, we discuss a way to find the \textit{exact} translation
length of an Anosov map on the curve graph of a \textit{torus}.

As noted before, $\Mod(S)$ acts on $\mathcal{C}(S)$ by isometry. In
particular, when $S$ is a torus, we have an isometry
$\SL(2,\mathbb{Z})$-action on $\mathcal{F}$. In fact, we can embed
$\mathcal{F}$ into $\hp$, so that $\SL(2,\mathbb{Z})$-action on
$\mathcal{F}$ can be seen as a restriction of the
$\PSL(2,\mathbb{R})$-action on $\hp$, so-called M{\" o}bius
transformation. Moreover, an Anosov element in $\SL(2,\mathbb{Z})$ and
a hyperbolic element in
$\PSL(2,\mathbb{Z}) \subset \PSL(2,\mathbb{R})$ both are characterized
by their absolute value of trace being bigger than 2. It follows that
one can identify an Anosov element in $\SL(2,\mathbb{Z})$ with a
hyperbolic element in $\PSL(2,\mathbb{Z})$. Hence, in the rest of the
paper we will interchangeably say ``a \textit{hyperbolic} element of
$\PSL(2,\mathbb{Z})$'', and ``an \textit{Anosov} element of
$\SL(2,\mathbb{Z})$.''

Any hyperbolic element of $\PSL(2,\mathbb{Z})$ has the unique
invariant geodesic axis in $\hp$. This seems to suggest the existence
of an invariant bi-infinite geodesic in $\mathcal{F}$ associated with an
Anosov map. Indeed, this identification allowed us to prove the
following main theorem:
\begin{THME}
  For any Anosov map $f$, there exists a
  bi-infinite geodesic $\mathcal{P}$ in $\mathcal{F}$ on which $f$ acts transitively.
\end{THME}
We remark here that this is not the case for non-sporadic surfaces;
namely, a pseudo-Anosov mapping class may not have a geodesic axis in
the curve graph. In particular, when the genus is bigger than 2, this can be easily shown
using the fact due to \cite{kin2017small}; When $S=S_g$ with $g \ge 3$, the minimal stable translation length among pseudo-Anosov mapping classes in $\Mod(S_g)$ is bounded above $\frac{1}{g^2-2g-1}$, which is strictly less than $1$ for $g \ge 3$.

In particular, the stable translation length of every Anosov map on $\mathcal{F}$ is an
\textit{integer}. This can be seen as a slight strengthening of
Bowditch's result \cite{bowditch2008tight}---the stable translation length of
any pseudo-Anosov map on the curve graph is a \textit{rational number} with a bounded
denominator--- in a special case where the additional information
comes from the concrete description of the curve complex of the torus.
Also, we emphasize here that the proof of Theorem \ref{thm:geodAxis} is
constructive, so we can calculate the exact translation length of any
Anosov element.

In section \ref{sec:FareyGraphCF}, we review some of the standard facts on the Farey graph and continued fractions.
In section \ref{sec:LadderInFareyGraph}, we introduce a special subgraph, called a \textit{ladder}, of the Farey graph, which plays a key role in this paper. 
In section \ref{sec:PSLActionOnFareyGraph}, we look more closely at the $\PSL(2,\mathbb{Z})$-action on the Farey graph.
In section \ref{sec:TrLengthOnFareyGraph}, we finally derive our main theorem and provide a concrete way to calculate the exact translation length of an Anosov element on $\mathcal{F}$.
In section \ref{sec:appl}, we provide three applications of our main theorem. Within a specific class of words in $\PSL(2,\mathbb{Z})$, we decide which form of words realizes the minimal translation length on $\mathcal{F}$(Theorem \ref{thm:minimal_length}).
Inspired from \cite{Gadre2013Lipschitz}, \cite{aougab2015pseudo} and \cite{baik2017typical}, we establish similar results for the translation length of an Anosov map on the Farey graph, namely the effective bound on the ratio of translation lengths of Anosov map on the Farey graph to that on Teichm{\" u}ller space(Theorem \ref{thm:boundsOfRatio}), and the overall growth of the number of Anosov maps which have the same translation lengths(Theorem \ref{thm:typical}).
In Appendix, we provide algorithms for generating and calibrating ladder. Both are
crucial for calculating the translation length of an Anosov element.

\begin{ACK}
  We are grateful to Mladen Bestvina, John Blackman, Sami Douba,
  Hongtaek Jung, Kyungro Kim and Dan Margalit for helpful
  comments. This work was supported by 2018 Summer-Fall KAIST
  Undergraduate Research Program. The first author was partially
  supported by Samsung Science \& Technology Foundation grant
  No. SSTFBA1702-01. The fourth author was supported by Basic Science
  Research Program through the National Research Foundation of
  Korea(NRF) funded by the Ministry of Education
  (NRF-2017R1D1A1B03035017).
\end{ACK}

\section{Farey Graph and Continued Fraction}
\label{sec:FareyGraphCF}

We review some facts of the Farey graph and continued fractions.
\begin{DEF}
  The \textbf{Farey graph} is a simplicial graph
  where each vertex is an extended rational number denoted by
  $\frac{p}{q}$, and a pair of vertices is joined by an edge if and only if
  these two vertices represent $\frac{p}{q}$ and $\frac{r}{s}$ satisfying
  $|ps-qr|=1$.
\end{DEF}
Denote by $\hp$ the hyperbolic plane. We can naturally embed Farey
graph $\mathcal{F}$ into a compactification of the hyperbolic plane
$\overline{\mathbb{H}}=\hp \cup \partial\hp$, where the vertices of
$\mathcal{F}$ are in correspondence with \textbf{extended rational
  points}
$\overline{\mathbb{Q}} = \mathbb{Q} \cup \{\frac{1}{0}=\infty\}
\subset \overline{\mathbb{R}} = \mathbb{R} \cup \{\infty\}$, and the
edges are represented by hyperbolic geodesics. Then $\hp$
is completely partitioned by the ideal triangles whose sides are the
edges of the Farey graph. We call those triangles as \textbf{Farey
  triangles}.  In the rest of the paper, we will regard the Farey graph as
an embedded graph in $\overline{\mathbb{H}}$.

\begin{DEF}
  A \textbf{positive(negative) continued fraction} is a continued
  fraction with positive(negative, respectively) integral
  coefficients.  A \textbf{periodic continued fraction} is an infinite
  continued fraction whose coefficients eventually repeat. A continued fraction is
  denoted by the following notation:
  \[
    [a_0;a_1,a_2,a_3,\cdots,a_n] = a_0 + \frac{1}{a_1+\frac{1}{a_2+\frac{1}{\cdots+\frac{1}{a_n}}}}.
  \]
  Also, we introduce the \textit{bar notation} to represent periodic continued
  fractions:
  \[
    [a;b,\overline{x,y,z}]=[a;b,x,y,z,x,y,z,x,y,z,\cdots].
  \]
\end{DEF}

\begin{RMK}
  \begin{enumerate}
  \item Every negative continued fraction can be written as a negation
    of a positive continued fraction:
    \[
      [-b_0;-b_1,-b_2,\cdots,-b_n] = -[b_0;b_1,b_2,\cdots,b_n].
    \]
  \item Every positive(negative) real number has a positive(negative, respectively) continued fraction representation. For example:
    \[
      \frac{5+\sqrt{17}}{6} = [1;1,\overline{1,11,1,2,5,1,5,2}], \quad \frac{5-\sqrt{17}}{6}=[0;6,\overline{1,5,2,1,11,1,2,5}].
    \]
  \item Lagrange showed that a continued fraction is periodic if and only if
it represents a \textit{quadratic irrational}. (See \cite{steinig1992proof}, for example.)
  \end{enumerate}
\end{RMK}


We now introduce the \textit{cutting sequence} of a geodesic in $\hp$.
A cutting sequence links the Farey graph with continued
fractions.  The following definition and Proposition
\ref{prop:CuttingAndPCF} are excerpted from \cite{series2015continued}.

\begin{DEF}[\cite{series2015continued}]
  The \textbf{cutting sequence} of $x \in \mathbb{R}$ is a
  sequence of $L$'s and $R$'s
  which are constructed in the following way:

  Join $x$ to any point on the imaginary axis $\mathbb{I}$ by a
  hyperbolic geodesic $\gamma$. This arc in $\hp$ cuts a succession of
  Farey triangles, which are ideal triangles, so $\gamma$ cuts exactly
  two edges of each Farey triangle. Then the two edges meet in a
  vertex $v$ on the left or right of the oriented geodesic $\gamma$. Label
  this vertex $v$ with $L$ or $R$ accordingly. (In the
  exceptional case in which $\gamma$ terminates in a vertex of the
  triangle, \textit{choose either $L$ or $R$}.) Then the resulting
  sequence $L^{n_0}R^{n_1}L^{n_2}\cdots$ is defined to be the cutting
  sequence of $x$.
\end{DEF}

\begin{RMK}
 The cutting sequence is independent of the choice of the initial point of $\gamma$ on $\mathbb{I}$.
\end{RMK}

The following proposition shows cutting sequences are closely
related to continued fractions.

\begin{PROP}[\cite{series2015continued}]
  \label{prop:CuttingAndPCF}
  Suppose $x>1$. Then $x$ has a cutting sequence as $L^{n_0}R^{n_1}L^{n_2}\cdots$, $n_i \in \mathbb{N}$
  if and only if $x = [n_0;n_1,n_2,\cdots]$, written in the positive continued fraction.
  Likewise, when $0<x<1$, $x$ has a cutting sequence $R^{n_1}L^{n_2}\cdots$, $n_i \in \mathbb{N}$,
  if and only if $x=[0;n_1,n_2,\cdots]$.
\end{PROP}

Therefore, \textit{the exponents in the cutting sequence of $x$ are the coefficients of the continued fraction representation of $x$}.

Since $x \mapsto -1/x$ is a half-turn in $\hp$ in the disk model,
the cutting sequence of $x$ and that of $-1/x$ are identical.
Using this observation, we can extend Proposition \ref{prop:CuttingAndPCF} to a situation with $x<0$.

\begin{COR}
  \label{cor:CuttingAndNCF}
  Suppose $x<-1$. Then $x$ has a cutting sequence $R^{n_1}L^{n_2}\cdots$, $n_i \in \mathbb{N}$
  if and only if $x = [-n_1;-n_2,\cdots]$, written in the negative continued fraction.
  Likewise, when $-1<x<0$, $x$ has a cutting sequence $L^{n_0}R^{n_1}L^{n_2}\cdots$, $n_i \in \mathbb{N}$,
  if and only if $x = [0;-n_0,-n_1,-n_2,\cdots]$, written in the negative continued fraction.
  In particular, for $x<0$, the coefficients of the positive continued fraction of $-1/x$
  and those of the negative continued fraction of $x$
  are identical up to translation of coefficients.
 \end{COR}

 More generally, we can define a \textbf{cutting sequence} of an
 \textit{oriented bi-infinite geodesic} in $\hp$ as a bi-infinite sequence of
 $L$'s and $R$'s. To be precise, pick a point $x$ on a bi-infinite
 geodesic $g$, and split $g$ at $x$ into two geodesic rays $g_1$ and
 $g_2$ with induced orientation. We may assume that $g_1$ terminates
 at $x$, and $g_2$ begins at $x$.  Then the desired bi-infinite
 sequence is obtained by concatenating \textit{two cutting sequences}:
 The $L$-$R$ flipped sequence of the cutting sequence of the reverse of
 $g_1$, followed by the cutting sequence of $g_2$.

 As the beginning point of a geodesic only affects the initial part of
 a cutting sequence, two oriented geodesics with the same endpoints
 \textit{eventually} have the same cutting sequence: The following
 lemma will not be needed until Theorem \ref{thm:periodicBiInfiniteLadder}.
\begin{LEM}[\cite{series2015continued}]
  \label{lem:cuttingSeqCoincides}
  Let $\gamma$, $\gamma'$ be oriented geodesics in $\hp$ with the same endpoint.
  Then the cutting sequence of $\gamma$ and $\gamma'$ eventually coincide.
\end{LEM}

\begin{proof}
  Refer to \cite{series2015continued}.
\end{proof}

\section{Ladder in Farey Graph}
\label{sec:LadderInFareyGraph}

\subsection{Ladder}
\label{ssec:Ladder}

In this section, we introduce a special class of subsets of Farey
graph, called \textit{ladders}. A ladder is a useful tool to analyze
geodesics in the Farey graph. Later on, we will see that a ladder is
\textit{geodesically convex}, so it captures all the geodesics joining
two vertices in $\mathcal{F}$.  Hatcher also introduced a ladder with
different terminology(`\textit{fan}') in
\cite{hatcher2017topology}. With ladders, he studied symmetries of
the Farey graph, which are applied to prove number theoretic results. We
focus on the dynamics of $\SL(2,\mathbb{Z})$ on the Farey graph.  To this end,
we establish a ladder which is stabilized by the action of a given
Anosov element.

\begin{DEF}
  Let $g$ be an oriented bi-infinite geodesic in $\hp$. Then we define
  the \textbf{ladder} associated with $g$ as the collection of all Farey
  triangles whose interior intersects with $g$.  Moreover, for any
  real numbers $x,y \in \partial\hp$, we define the \textbf{ladder}
  associated with $x,y$, denoted by $\mathcal{L}(x,y)$, as the ladder
  associated with the bi-infinite oriented geodesic joining from $x$ to $y$.
\end{DEF}

\begin{RMK}
  From the above definition, we specified the \textit{interior} of a Farey triangle and
  a geodesic $g$ must intersect. Thus, if $x,y$ are vertices in
  $\mathcal{F}$ which form two endpoints of a single Farey edge, then
  $\mathcal{L}(x,y) = \phi$.
\end{RMK}

Since a geodesic cannot pass all three sides of a single geodesic
triangle at once, we can characterize a ladder as following:

\begin{FACT}
  \label{fact:chrLadder}
  A ladder is a \emph{consecutive chain} of Farey triangles; i.e., a countable union
  of Farey triangles $\{\Delta_i\}$ such that $\Delta_i \cap \Delta_{i+1}$ is a single
  edge of the Farey graph and $\Delta_i \cap \Delta_j$ is either an empty set or a single point(which will be called as a pivot point, shortly.) if $|i-j| \ge 2$.
\end{FACT}

From this fact, now we can further define each component of a ladder:
\begin{DEF}  
\begin{figure}[h]\centering    
  \includegraphics[width = .75\linewidth]{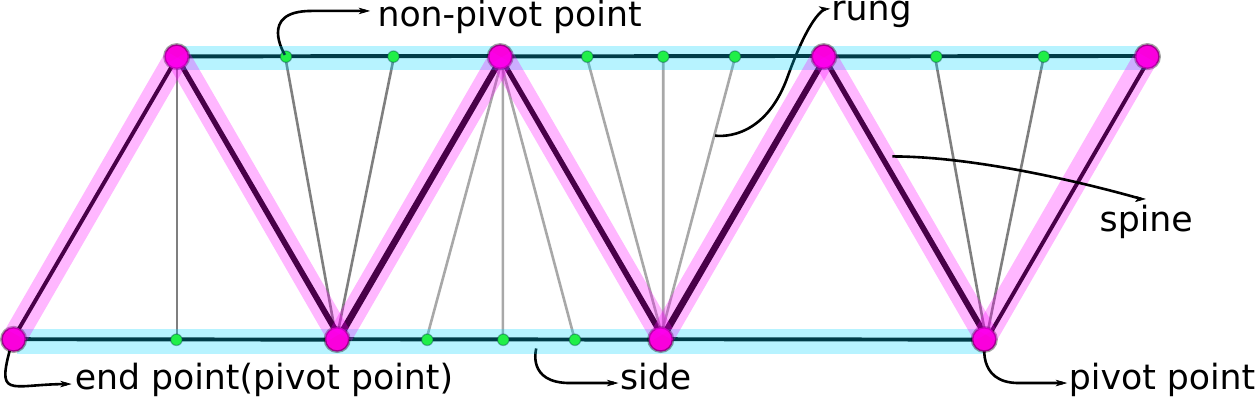}
  \caption[ladderDef]{Components of Ladder}
  \label{fig:ladder_def}
\end{figure}
An \textbf{endpoint} of a ladder is a degree 2 vertex of a ladder.
Connect two endpoints with an oriented geodesic $g$. While recording
the cutting sequence of $g$, call every $L$ or $R$-labeled vertex as
\textbf{a pivot point}.  Include two endpoints as pivot points as
well. An edge in a ladder is called a \textbf{rung} if its
\textit{interior} and $g$ intersect.

The \textbf{spine} $K$ of a ladder $\mathcal{L}$ is the path in $\mathcal{L}$ with the following property:
\begin{itemize}
\item The beginning point and the terminal point of $K$ are endpoints of $\mathcal{L}$.
\item All the vertices of $K$ are \textit{exactly all the pivot points} in $\mathcal{L}$.
\item All the edges of $K$ except for the initial and final one are rungs of $\mathcal{L}$.
\end{itemize}
For the uniqueness of a spine for each ladder, we define a spine for the following exceptional case as Figure \ref{fig:1-1_ladder}. See the remark below for the uniqueness of other ladders.

\begin{figure}[h]\centering    
  \includegraphics[width =0.3\linewidth]{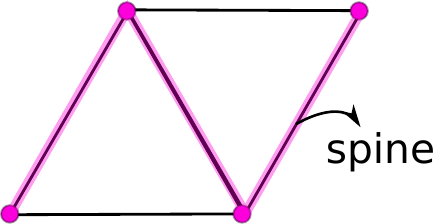}
  \caption[ladderType]{Spine for $(1,1)$-type ladder.}
  \label{fig:1-1_ladder}
\end{figure}

Then the \textbf{side} of a ladder is a connected component of the union of all non-rung edges which do not belong to the spine of the ladder.
See Figure \ref{fig:ladder_def}.
\end{DEF}

\begin{RMK}
  In fact, if a ladder is other than $(1,1)$-type, then its spine is
  uniquely determined.  This is because except for the $(1,1)$-type
  ladder, we can connect each pivot point to another pivot point via the
  unique rung, except two endpoints. After joining them via rungs, we
  get a path $\mathcal{P}'$ in the ladder except for two endpoints.
  Now we can attach two endpoints to $\mathcal{P}$ in a
  unique way. See Figure \ref{fig:spine_uniqueness} for the
  illustration of this construction of a spine.
\end{RMK}

\begin{figure}[h]\centering    
  \includegraphics[width =\linewidth]{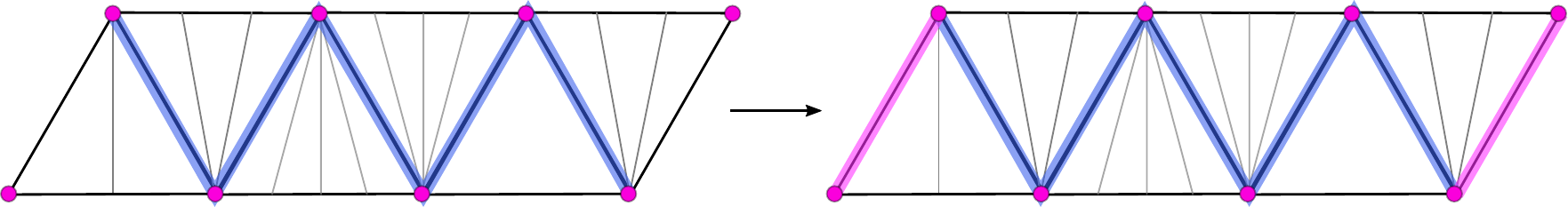}
  \caption[ladderType]{How spine can be formed uniquely: Connect each pivot point to another point via the unique rung(left) and then attach two endpoints to the path(right).}
  \label{fig:spine_uniqueness}
\end{figure}

We can specify the \textit{type} of a ladder associated with a geodesic by counting a sequence of
consecutive Farey triangles sharing the same pivot point.
\begin{DEF}
  
\begin{figure}[h]\centering    
  \includegraphics[width =\linewidth]{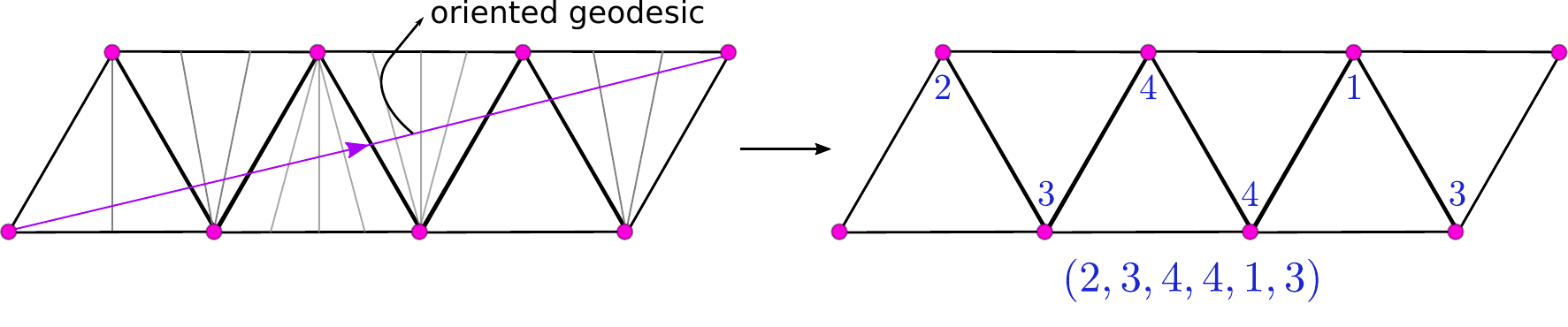}
  \caption[ladderType]{Type of Ladder}
  \label{fig:ladder_type}
\end{figure}
We say a ladder is of \textbf{type} $(a_1,\cdots,a_n)$ if the
ladder has $a_1,\cdots,a_n$ consecutive numbers of Farey triangles with
the same pivot points read off in the orientation given to the
geodesic. As an example, see Figure \ref{fig:ladder_type}. In this case, we call $n$ as the
\textbf{length} of the ladder.
\end{DEF}


  Immediate from the definitions, the number of Farey triangles
  sharing the same pivot points should coincide with the cutting
  sequence associated with $g$.
\begin{PROP}
  \label{prop:cuttingSeqAndLadder}
  Let $g$ be an oriented bi-infinite geodesic in $\hp$. Then
  the \emph{type} of the ladder associated with $g$ is identical to the \emph{exponents} of
  the cutting sequence of $g$.
\end{PROP}

\subsection{Geodesics in Ladder}
\label{ssec:GeodesicsInLadder}
For two points $x,y \in \partial\overline{\mathbb{H}}$, define $E(x,y)$ as the set of all edges in $\mathcal{F}$ separating $x,y$.
This definition of $E(x,y)$ generalizes that of \cite{minsky1996geometric}, in which $E(x,y)$ is only defined
for $x,y$ being vertices of the Farey graph. This slightly generalized approach is to handle bi-infinite geodesics.
\begin{PROP}
  \label{prop:E-EdgesAreRungs}
 All edges in $E(x,y)$ are in one-to-one correspondence with the rungs in $\mathcal{L}(x,y)$.
\end{PROP}

\begin{proof}
  Denote by $\overline{xy}$ the bi-infinite geodesic connecting $x$ and $y$ in $\hp$.
  Firstly, pick an edge $e \in E(x,y)$. As $e$ separates $\hp$ into 2 components, $x,y$ fall in
  different parts of them. Thus the interior of $e$ must intersect with $\overline{xy}$.
  Therefore, $e$ must be a rung of $\mathcal{L}(x,y)$.

  Conversely, pick any rung $r$ from $\mathcal{L}(x,y)$. Then there
  are exactly two Farey triangles incident with $r$, and they meet
  $\overline{xy}$ by the definition of $\mathcal{L}(x,y)$. Since those
  two triangles reside in different components separated by $r$, it
  follows that $x,y$ also must be separated by $r$. Therefore,
  $r \in E(x,y)$.
\end{proof}

When it comes
to considering geodesics in the whole Farey graph, Proposition \ref{prop:E-EdgesAreRungs} and the following
fact from \cite{minsky1996geometric} justify the reason why it
suffices to consider only the geodesics in ladders.

\begin{FACT}[\cite{minsky1996geometric}]
  \label{fact:Minsky}
  Let $x,y$ be two vertices of the Farey graph. Then for any geodesic path
  $\mathcal{P}$ joining $x$ and $y$ in $\mathcal{F}$, each vertex of
  $\mathcal{P}$ other than $x$ and $y$ must be incident with some
  edge in $E(x,y)$.
\end{FACT}

\begin{COR}[Ladder is Geodesically Convex]
  \label{cor:LgeodIffFgeod}
  Let $x,y$ be two vertices of the Farey graph,
  and $\mathcal{P}$ be a path joining $x$ and $y$ in $\mathcal{F}$.
  Then $\mathcal{P}$ is a geodesic in $\mathcal{F}$
  if and only if $\mathcal{P}$ is a geodesic in the ladder $\mathcal{L}(x,y)$.
  In other words, a ladder is \emph{geodesically convex} in $\mathcal{F}$.
\end{COR}

\begin{proof}
  Suppose $\mathcal{P}$ is a geodesic in $\mathcal{F}$.  Then by Fact
  \ref{fact:Minsky}, all the vertices in $\mathcal{P}$ must be
  incident with edges in $E(x,y)$, which are rungs of
  $\mathcal{L}(x,y)$ due to Proposition \ref{prop:E-EdgesAreRungs}.
  Thus, every vertex of $\mathcal{P}$ is contained in
  $\mathcal{L}(x,y)$, so $\mathcal{P}$ is a path in
  $\mathcal{L}(x,y)$.  It is a geodesic in $\mathcal{L}(x,y)$ as well,
  otherwise there is a shorter path than $\mathcal{P}$ connecting $x$
  and $y$ in $\mathcal{L}(x,y)$, thus in $\mathcal{F}$, which
  contradicts to the assumption that $\mathcal{P}$ is a geodesic in
  $\mathcal{F}$. Thus, $\mathcal{P}$ is a geodesic in
  $\mathcal{L}(x,y)$.

  Conversely, let $\mathcal{P}$ be a geodesic in $\mathcal{L}(x,y)$.
  Suppose there is a shorter path $\mathcal{P}'$ than $\mathcal{P}$ in
  $\mathcal{F}$ joining $x$ and $y$. Then by the same argument,
  $\mathcal{P}'$ is contained in $\mathcal{L}$, still shorter
  than $\mathcal{P}$, a contradiction. Therefore, $\mathcal{P}$ is a
  geodesic in $\mathcal{F}$.
\end{proof}

Here we prove a useful lemma which hints the \textit{hierarchy} between ladders.

\begin{LEM}
  \label{lem:SubladderLemma}
  Let $x,y$ be points in the boundary of $\hp$, and $r,s$ be two vertices of $\mathcal{L}(x,y)$.
  Then $\mathcal{L}(r,s)$ is contained in $\mathcal{L}(x,y)$.
\end{LEM}

\begin{proof}
  It suffices to show every rung in $\mathcal{L}(r,s)$ is a rung of $\mathcal{L}(x,y)$,
  since non-rung edges of a ladder are fully determined with rungs by completing Farey triangles.
  
  Pick a rung $e$ from $\mathcal{L}(r,s)$. According to
  Proposition \ref{prop:E-EdgesAreRungs}, it suffices to show that $e$
  separates $x,y$. Let $A, B$ be two connected components of
  $\hp \setminus e$. Since $e$ is a rung of $\mathcal{L}(r,s)$, $e$ is
  disjoint from $r,s$. Thus each $A, B$ contains exactly one of $r,s$.
  Now suppose to the contrary $x,y$ are not separated by $e$. Then the
  interior of the bi-infinite geodesic $\overline{xy}$ in $\hp$ must
  reside in either $A$ or $B$. Hence, $\mathcal{L}(x,y)$ is
  contained in either $\overline{A}$ or $\overline{B}$, since two
  edges of the Farey graph cannot intersect in interior points. In
  particular, both vertices $r,s$ of $\mathcal{L}(x,y)$ should be 
  contained together in either $A$ or $B$, which is a contradiction.
\end{proof}

Recall that a \textbf{bi-infinite geodesic} in $\mathcal{F}$ is a path
in $\mathcal{F}$ whose every finite subpath is also a geodesic in
$\mathcal{F}.$ Note that a ladder associated with a bi-infinite geodesic
in $\mathcal{F}$ is bi-infinite as well.

Likewise, one can define a bi-infinite geodesic in a bi-infinite ladder.
It is Lemma \ref{lem:SubladderLemma} that makes this definition allowable.
\begin{DEF}
  Let $\mathcal{L}$ be a bi-infinite ladder in $\mathcal{F}$.  A
  bi-infinite path $\mathcal{P}$ in $\mathcal{L}$ is called a
  \textbf{bi-infinite geodesic} if for every pair of vertices $s,t$ of
  $\mathcal{P}$, the restriction $\mathcal{P}(s,t)$ of $\mathcal{P}$ to the subladder
  $\mathcal{L}(s,t)$ is again a geodesic in $\mathcal{L}$.
\end{DEF}

\begin{RMK}
  Since $\mathcal{L}(s,t)$ is geodesically convex by Corollary \ref{cor:LgeodIffFgeod}, the last statement of above definition can be reduced to showing $\mathcal{P}(s,t)$
  is a geodesic in $\mathcal{L}(s,t)$, instead of $\mathcal{L}$.
\end{RMK}

With putting more effort on the proof of Corollary \ref{cor:LgeodIffFgeod},
we can extend Corollary \ref{cor:LgeodIffFgeod} to general result involving bi-infinite objects.

\begin{PROP}
  \label{prop:GenLgeodIffFgeod}
  Let $x,y$ be \emph{irrational} points in $\partial\hp$(i.e., points
  in $\partial\hp$ which are not vertices of $\mathcal{F}$), and
  $\mathcal{P}$ be a bi-infinite path joining $x,y$ in the Farey graph.
  Then $\mathcal{P}$ is a geodesic in $\mathcal{F}$ if and only if
  $\mathcal{P}$ is a geodesic in the bi-infinite ladder
  $\mathcal{L}(x,y)$.
\end{PROP}

\begin{proof}
  Let $\mathcal{P}$ be a bi-infinite path joining $x$ and $y$ in
  $\mathcal{L}(x,y)$. For any two vertices $s,t$ in $\mathcal{P}$,
  denote by $\mathcal{P}(s,t)$ the induced subpath joining $s$ and
  $t$.

  Then by definition $\mathcal{P}$ is a bi-infinite geodesic in $\mathcal{L}(x,y)$
  if and only if $\mathcal{P}(s,t)$ is a geodesic in
  $\mathcal{L}(s,t)$ for any two vertices $s,t$ in $\mathcal{P}$.  By
  Corollary \ref{cor:LgeodIffFgeod}, This is equivalent to saying
  $\mathcal{P}(s,t)$ is a geodesic in $\mathcal{F}$ for any two
  vertices $s,t$ of $\mathcal{P}$ and by definition it is exactly the case when
  $\mathcal{P}$ is a bi-infinite geodesic in $\mathcal{F}$.
\end{proof}

\subsection{Efficient Geodesics in Ladder}
\label{sec:EffGeodesicsInLadder}



Let $g$ be a bi-infinite oriented geodesic in $\hp$ and $\mathcal{L}$
be its associated ladder. As a consecutive chain of Farey triangles,
we give the ladder $\mathcal{L}$ an orientation induced by the
orientation of $g$.  Then the Farey triangles and pivot points of
$\mathcal{L}$ have the induced order: Enumerate them as Figure \ref{fig:pivotOrder_ladder}.

Then for each $i$-th pivot point, the $(i+1)$-th one is called the
\textbf{adjacent pivot point on the other side} and the $(i+2)$-th one
is called the \textbf{adjacent pivot point on the same side}.
Those $(i+1)$-th and $(i+2)$-th pivot points are \textbf{adjacent} pivot points of $i$-th one.
Moreover, call the first pivot point as  \textbf{initial} and the last one as \textbf{final},
whose immediate predecessor is called \textbf{semi-final}. In other words,
the semi-final pivot point has the adjacent final point on the other side.
  
\begin{figure}[h]\centering    
  \includegraphics[width = .8\linewidth]{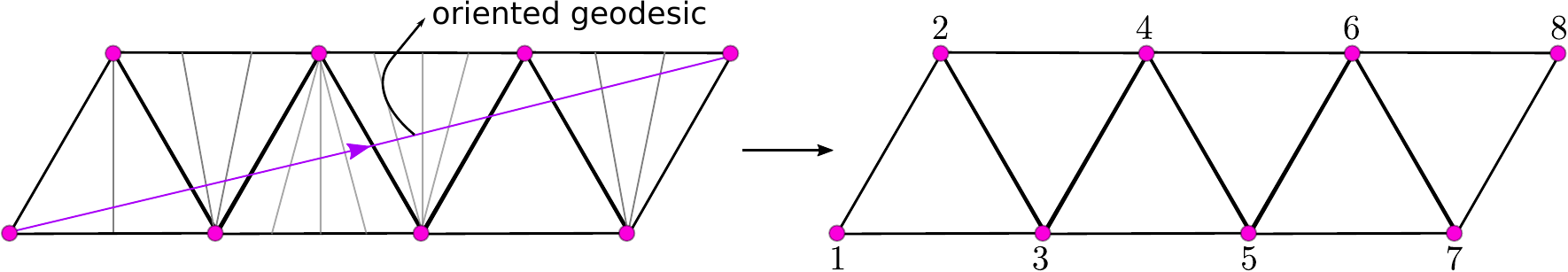}
  \caption[pivotOrderLadder]{An orientation of a ladder induces an
    order of pivot points. The point 6 is the adjacent
    pivot point on the same side of the pivot point 4, which is the adjacent pivot point
    on the other side of point 3.}
  \label{fig:pivotOrder_ladder}
\end{figure}

Now we define two key moves across pivot points, which will be a
building block for efficient moving.

\begin{DEF}
For each non-terminal pivot point, define the following two moves from
a pivot point to the next adjacent pivot point:
\begin{enumerate}
\item \textbf{Transverse}(`$t$'): Move to the adjacent pivot point on the other
  side.
\item \textbf{Pass}(`$p$'): Move to the adjacent pivot point on the same side.
\end{enumerate}

\begin{figure}[h]\centering    
  \includegraphics[width = .4\linewidth]{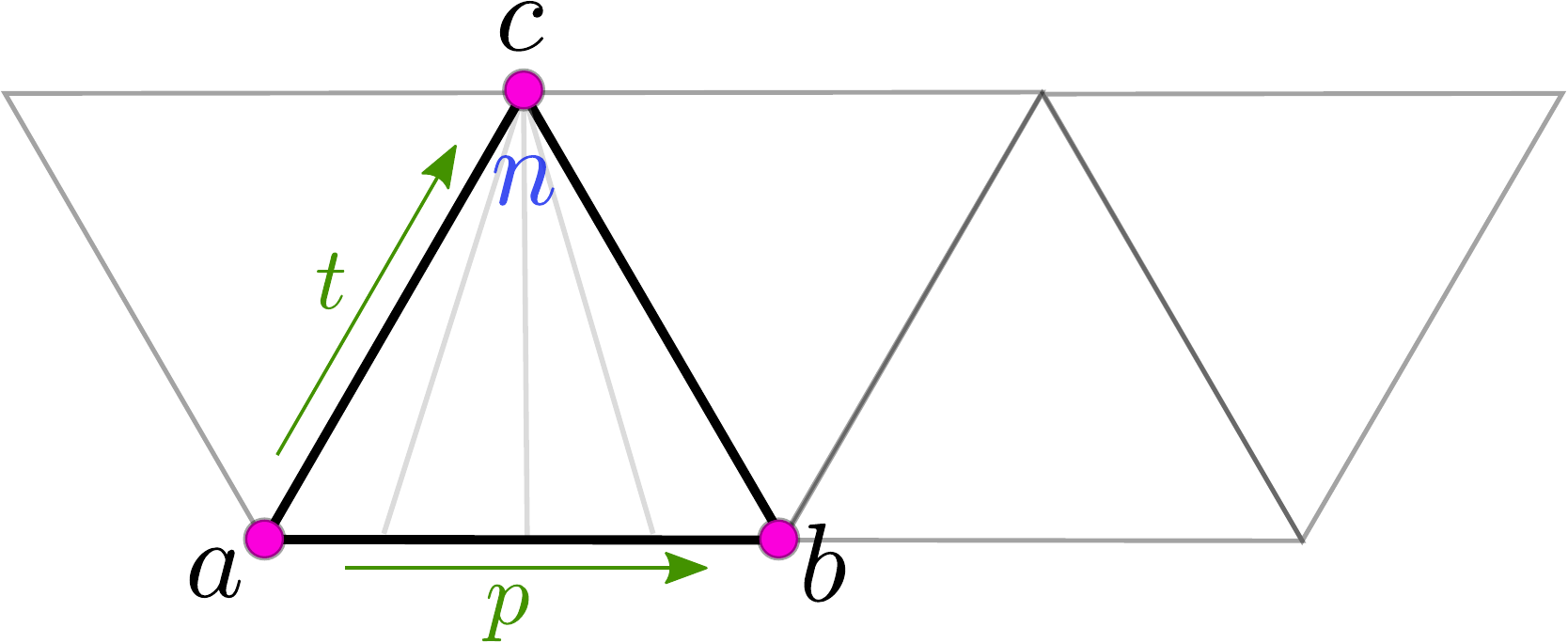}
  \caption[effGeodLadder]{The pivot point $a$ has two options to move toward adjacent pivot points: $b,c$.}
  \label{fig:effGeod_ladder}
\end{figure}

We say a path $\mathcal{P}$ in a ladder $\mathcal{L}$ satisfies the
\textbf{efficient moving condition} if for each move between adjacent
pivot points in $\mathcal{P}$,
  \begin{itemize}
  \item Move $t$ whenever the move $p$ passes through more than one
    edge in $\mathcal{L}$. ($n \ge 2$ in Figure
    \ref{fig:effGeod_ladder}), \textbf{or} the final point can be reached through the move $t$.
  \item Move $p$ whenever the move $p$ passes exactly one edge in
    $\mathcal{L}$ ($n = 1$ in Figure \ref{fig:effGeod_ladder}), \textbf{and} the final point
    of cannot be reached through the move $t$.
  \end{itemize}
  See Figure \ref{fig:effGeod_Examples} for more illustrations on the efficient moving condition.
\end{DEF}

\begin{figure}[p]
\begin{subfigure}{.5\textwidth}
  \centering
  \includegraphics[width=.8\linewidth]{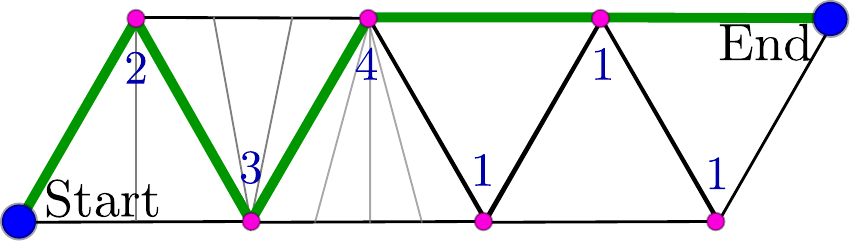}
  \caption{$tttpp$}
  \label{fig:effGeod_eg1}
\end{subfigure}%
\begin{subfigure}{.5\textwidth}
  \centering
  \includegraphics[width=.8\linewidth]{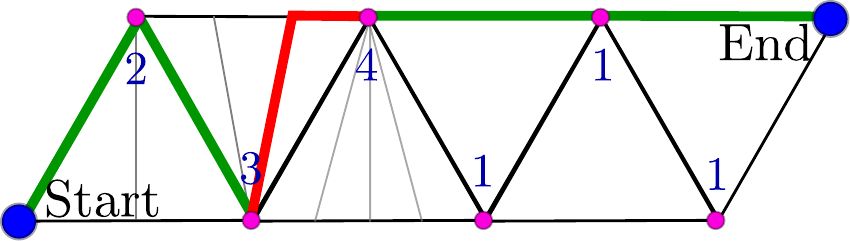}
  \caption{Forbidden path}
  \label{fig:effGeod_eg2}
\end{subfigure}%
\vspace{2em}
\begin{subfigure}{.5\textwidth}
  \centering
  \includegraphics[width=.8\linewidth]{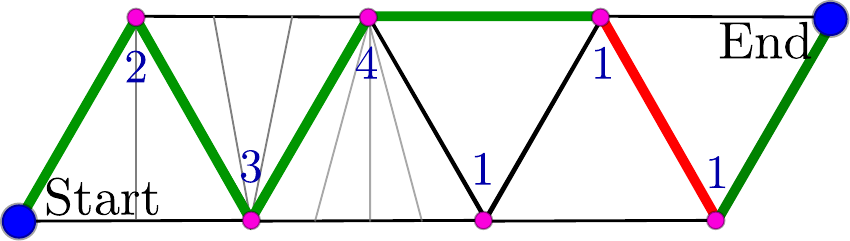}
  \caption{$tttptt$}
  \label{fig:effGeod_eg3}
\end{subfigure}%
\begin{subfigure}{.5\textwidth}
  \centering
  \includegraphics[width=.8\linewidth]{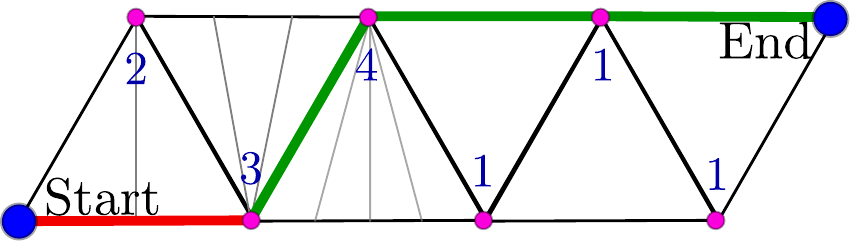}
  \caption{$ptpp$}
  \label{fig:effGeod_eg4}
\end{subfigure}%
\vspace{2em}
\begin{subfigure}{.5\textwidth}
  \centering
  \includegraphics[width=.8\linewidth]{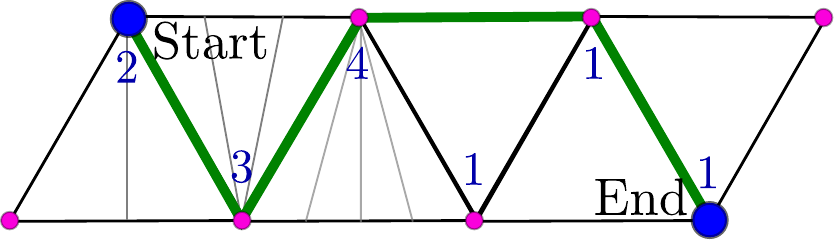}
  \caption{$ttpt$}
  \label{fig:effGeod_eg5}
\end{subfigure}%
\begin{subfigure}{.5\textwidth}
  \centering
  \includegraphics[width=.8\linewidth]{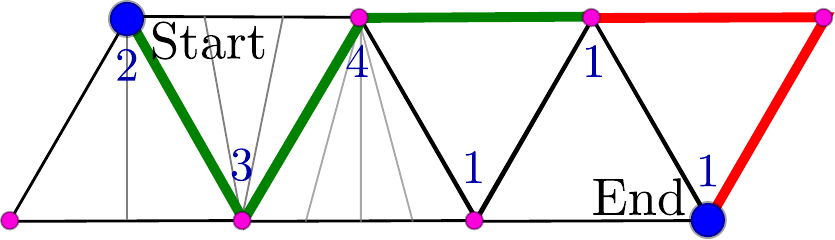}
  \caption{Forbidden path}
  \label{fig:effGeod_eg6}
\end{subfigure}%
\caption{Six paths in a $(2,3,4,1,1,1)$-ladder. Two blue points for
  each figure indicate the beginning points and the final points of
  paths. Green segments stand for efficient moves, and red ones for
  non-efficient moves. Only $(a)$ and $(e)$ satisfy the efficient
  moving condition. $(b)$ commits forbidden moves; the path contains a
  non-pivot point. $(c)$ fails to satisfy the efficient moving
  condition at $1$-block near the end; it should have moved
  $p$. Similarly, $(d)$ should have moved $t$ at the beginning; note
  that, however, it \textit{is} a geodesic, so some geodesics may not meet
  the efficient moving condition. (See the remark after the proof of Proposition
  \ref{prop:EffmovingIsGeod}.) The problem in $(f)$ is regarding the
  final point. It fails to move $t$ at the end of the path; even
  though the next one is a 1-block, it should have moved $t$ since the
  final point could be reached through the move $t$.}
\label{fig:effGeod_Examples}
\end{figure}

Observe that the efficient moving condition is indeed \textit{efficient} in a sense of the following proposition: It yields a geodesic.

\begin{PROP}
  \label{prop:EffmovingIsGeod}
  If a path in a ladder satisfies the efficient moving condition then
  it is a geodesic in the ladder.
\end{PROP}

\begin{proof}

  Let $\mathcal{L}$ be a ladder and $a$ be a pivot point which is
  neither a final point nor a semi-final point of $\mathcal{L}$.
  (Such exceptional cases are obvious; the final point has no option
  to move. The semi-final point has only one option: $t$. These
  options satisfy the efficient moving condition.)  We will show that
  the efficient moving condition forces $a$ to move while minimizing
  the length of its locus. Then this eventually implies a path
  $\mathcal{P}$ in $\mathcal{L}$ satisfying the efficient moving
  condition should be a geodesic.

  Recall that $a$ has only two options to move to adjacent pivot points: pass or transverse.
  (Otherwise, it backtracks, i.e., it moves to pivot points with the lower order, which hinders the path to be a geodesic.)
  Denote by $b$, $c$ the adjacent pivot points of $a$ on the same and the other side respectively.
  Also, denote by $n$ the number of Farey triangles in $\triangle abc$.
  See Figure \ref{fig:effGeod_ladder}. We break into two cases:

\begin{itemize}
\item $\mathbf{n=1}$. If $b$ is the endpoint, then $p$ is shorter than
  $tt$. Thus, $a$ must choose to move $p$.  If $b$ is not the
  endpoint, let $d$ be the adjacent pivot point of $b$ on the other
  side.  Then $a$ must visit $b$ or $d$. If $a$ visits $b$, then $p$
  is the shortest option for $a$ to choose. If $a$ visits $d$, then
  $pt$ is always the shortest option for $a$ to choose, no matter how
  many Farey triangles are in $\triangle bcd$. In any case, $a$ can
  choose $p$ to minimize the length of its locus, which aligns with
  the efficient moving condition.
\item $\mathbf{n\ge 2}$. Note that $a$ must visit $b$ or $c$.  If $a$
  visits $b$, then there are two possible options: $tt$ or $p$.
  However, since now $p$ leaves at least length 2 trail, $tt$ is
  always the shortest path for $a$ to travel toward $b$. If $a$ visits
  $c$, then obviously $t$ is the shortest choice.  In any case, $a$
  can choose $t$ to minimize the length of its locus, which aligns
  with the efficient moving condition.
\end{itemize}
All in all, it follows that if $\mathcal{P}$ satisfies the efficient
moving condition, then the starting point of $\mathcal{P}$ must have
the minimum length of the locus by the above arguments.  Hence
$\mathcal{P}$ should be a geodesic in $\mathcal{L}$.
\end{proof}

\begin{RMK}
  \begin{enumerate}
  \item If the semi-final pivot point $z$ is contained in
    $\mathcal{P}$, the move $t$ is the only option for $z$. That is
    why the efficient moving condition contains the final point
    exceptions.
    
  \item As in the proof, the alternative path by allowing move $p$ in
    2-block does not generate a longer path, so it can be another
    geodesic. Thus we want to call a geodesic \textbf{efficient} if
    satisfies the efficient moving condition.
\end{enumerate}
\end{RMK}

For bi-infinite geodesics, the efficient moving condition is valid as well,
since the efficient moving condition is in some sense a \textit{local} condition.

\begin{PROP}
  If a bi-infinite path in a bi-infinite ladder satisfies the efficient moving condition, then it is a bi-infinite geodesic.
\end{PROP}

\begin{proof}
  By Proposition \ref{prop:EffmovingIsGeod}, every restricted finite
  subpath of a given path is a geodesic in the finite subladder. Hence by
  definition, it is a bi-infinite geodesic in the whole bi-infinite ladder.
\end{proof}

Now we narrow down our interest to bi-infinite geodesics in \textit{periodic} ladders.

\begin{DEF}
  Let $\mathcal{L}$ be an infinite(not necessarily bi-infinite)
  periodic ladder with a repeating pattern $(a_1,\cdots,a_n)$. If the
  pattern $(a_1,\cdots,a_n)$ is \textit{minimal} upto cyclic
  permutation, (i.e., it cannot be decomposed into copies of
  subpatterns upto cyclic permutation) then we say $\mathcal{L}$ is of
  \textbf{period} $(a_1,\cdots,a_n)$. In this case, a finite subladder of $\mathcal{L}$
  having the following type upto cyclic permutation:
\[
  \begin{cases}
    (a_1,\cdots,a_n) & \text{if $n$ is even,}\\
    (a_1,\cdots,a_n,a_1,\cdots,a_n) & \text{if $n$ is odd,}
  \end{cases}
\]
is called a \textbf{prime subladder} of $\mathcal{L}$.
\end{DEF}

\begin{RMK}
Observe that the length of any prime subladder is always \textit{even}.  
\end{RMK}
The following proposition says within a bi-infinite ladder $\mathcal{L}$,  we can establish bi-infinite efficient geodesic whenever $\mathcal{L}$ is \textit{periodic}, by concatenating
finite efficient geodesics in its prime subladder.
\begin{PROP}[Concatenating Geodesics]
  \label{prop:ConcatGeodesic}
  Let $\mathcal{L}$ be a periodic bi-infinite ladder with a prime
  subladder $\mathcal{L}'$.  Then there exists an efficient geodesic
  $\mathcal{P}'$ whose endpoints reside on the same side of
  $\mathcal{L}'$ such that by concatenating $\mathcal{L}'$
  bi-infinitely, the induced concatenated path $\mathcal{P}$ is a
  bi-infinite efficient geodesic $\mathcal{P}$ in $\mathcal{L}$.  More
  generally, if $\mathcal{L}''$ is a finite concatenation of prime
  subladders then there exists an efficient geodesic $\mathcal{P}''$ in
  $\mathcal{L}''$ whose induced bi-infinitely concatenated path
  $\mathcal{P}$ is a bi-infinite efficient geodesic in $\mathcal{L}$.
\end{PROP}

\begin{proof}
  Since we want to concatenate $\mathcal{P}'$ by juxtaposing $\mathcal{L}'$
  to make one single connected geodesic in $\mathcal{L}$,
  we need to consider \textit{maximal} efficient geodesics in $\mathcal{L}'$
  whose endpoints lying on the same side of $\mathcal{L}'$. Thus, there are
  always two efficient geodesics to consider in $\mathcal{L}'$,
  each starts and ends on the same side of $\mathcal{L}'$.

  Note the efficiency of a geodesic is determined by the coefficients
  of the ladder, except for the move from the semi-final point to the
  final point of the ladder, which we call \textbf{reluctant
    move}. This move is abnormal, as it is always determined to be $t$
  regardless of the succeeding coefficient of the ladder. In
  particular, even if the associated coefficient is 1, $\mathcal{L}'$
  fails to move $p$, which hinders $\mathcal{P}'$ to remain efficient
  when we embed $\mathcal{P}'$ as a subpath of $\mathcal{P}$.  See
  Figure \ref{fig:reluctant_move}.
  \begin{figure}[ht]
    \centering
    \includegraphics[width=\textwidth]{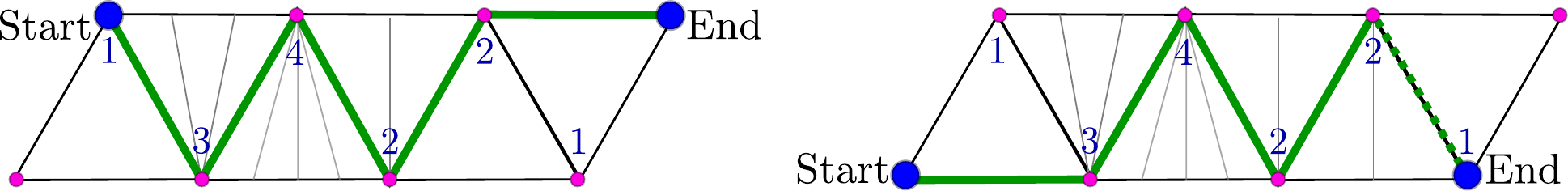}
    \caption{Illustration of Reluctant move. Within a ladder, there
      are two maximal efficient geodesics available whose endpoints lying on the same
      side. In this case, only the right one consists of
    the reluctant move, which is denoted by the dotted line.}
    \label{fig:reluctant_move}
  \end{figure}
  
  Therefore, the only possible obstruction for $\mathcal{P}$ to be
  efficient in $\mathcal{L}$ arises from the reluctant moves in
  $\mathcal{P}'$. Thus, now assume that two efficient geodesics in
  $\mathcal{L}'$ consist of reluctant moves, and suppose further to
  the contrary that these two moves indeed cause \textit{problems}
  during the bi-concatenation. By problems we mean the reluctant move
  cannot be extended as an efficient move after the
  bi-concatenation. In fact, the reluctant moves in two efficient
  geodesics should be as illustrated in Figure
  \ref{fig:reluctant_problematic}. To make each reluctant move
  problematic, \textit{both the first and last coefficients of
    $\mathcal{L}'$ must be 1.}

  \begin{figure}[ht]
    \centering
    \includegraphics[width=\textwidth]{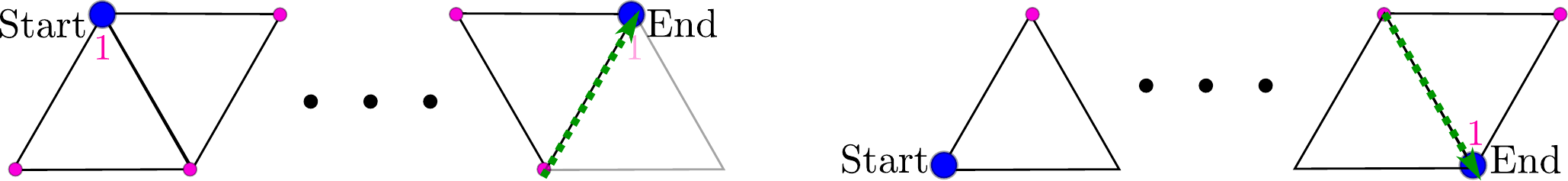}
    \caption{Both efficient geodesics have problematic reluctant moves(dotted lines). This forces
    the first and last coefficient of the ladder to be 1.}
    \label{fig:reluctant_problematic}
  \end{figure}
  To deduce a contradiction from this, we need the following lemma:

  \begin{LEM}
    \label{lem:twoGeodIntersect}
    Unless all the coefficients of a ladder are 1, the two efficient geodesics
    intersect at some point $x$ in the ladder, and follow the same
    pivot points starting from $x$, except for the final point.
  \end{LEM}
  \begin{proof}[Proof of Lemma \ref{lem:twoGeodIntersect}]
    \renewcommand{\qedsymbol}{$\triangle$} Suppose the two geodesics
    do not intersect. Then both should not contain $t$-moves because
    $t$ bisects the ladder so it makes the two geodesics inevitably
    meet each other. Therefore, both geodesics are of the form
    $p \cdots p$, in turn, all the coefficients of the ladder become
    1. Thus, we just showed the contraposition of the first
    assertion. Furthermore, note each efficient move at pivot points
    possibly except for the final point is completely determined by
    the corresponding coefficients of the ladder. Hence, once the two
    geodesics meet, they have to visit the exact same pivot points
    in $\mathcal{L}'$, possibly except for the final point.
  \end{proof}

  Now we are ready to prove our Proposition \ref{prop:ConcatGeodesic}.
  We claim that two geodesics in $\mathcal{L}'$ should not intersect.
  Suppose to the contrary that they intersect at some point in the ladder.
  Then two geodesics eventually visit the same pivot points. However,
  as in Figure \ref{fig:reluctant_problematic}, this forces one of two geodesics
  to move $tt$ at the final block of $\mathcal{L}'$, whose coefficient is 1, which contradicts
  to the efficiency condition.(It should have moved $p$ at the final block, instead of $tt$.)
  
  Thus, two geodesics must not meet each other; so their coefficient
  must be all 1's, by Lemma \ref{lem:twoGeodIntersect}. However, these
  consecutive $p$-moves forces both geodesics to avoid the semi-final
  points of the ladder, which contradicts to the assumption that both
  geodesics consist of reluctant moves.

  Therefore, we showed even if both of the two geodesics contain reluctant moves,
  at least one of them must not cause a problem during the bi-concatenation.
  All in all, we conclude at least one of two geodesics can be extended to
  $\mathcal{P}$ within $\mathcal{P}$.
\end{proof}

\section{$\PSL(2,\mathbb{Z})$-action on Farey Graph}
\label{sec:PSLActionOnFareyGraph}
Recall that $\PSL(2,\mathbb{Z})$ is isomorphic to the group of the
orientation preserving isometries of $\hp$. Indeed, there is an
isometric action $\PSL(2,\mathbb{Z})$ on $\mathcal{F} \subset \hp$, defined as
\[
  \begin{bmatrix}
    a & b \\ c & d
  \end{bmatrix}
  \cdot \frac{p}{q} = \frac{ap + bq}{cp + dq}
\]
on the vertices of $\mathcal{F}$, and induced action on the edges.
(See for instance \cite[Prop 3.1]{hatcher2017topology}) It follows
that any element $f$ of $\PSL(2,\mathbb{Z})$ maps a Farey triangle to
another Farey triangle. More generally, $f$ sends a ladder to the
ladder of the same type. This is because a ladder is a consecutive
chain of Farey triangles(Fact \ref{fact:chrLadder}), and a bi-infinite
geodesic $\gamma$ cuts a Farey triangle $T$ with the same symbol($L$
or $R$) as when $f(\gamma)$ cuts $f(T)$. (See the proof of \cite[Prop
2.2]{series2015continued}) Eventually, for a ladder $\mathcal{L}$ and
its associated geodesic $\gamma$, the pair $(\mathcal{L},\gamma)$ yields the
same cutting sequence as $(f(\mathcal{L}),f(\gamma))$ does. Thus, we
proved the following proposition:
\begin{PROP}
  \label{prop:preservesLadder}
  $\PSL(2,\mathbb{Z})$ preserves the type of a ladder.
\end{PROP}

  

For a hyperbolic isometry $f$ in $\PSL(2,\mathbb{Z})$, denote by $A_f$ the axis of $f$ in $\hp$.
Recall that $A_f$ is defined to be the bi-infinite geodesic joining the two fixed points of $f$,
which are represented by conjugate quadratic irrational numbers.
It is also a basic fact that $f$ acts on $A_f$ by translation.

Now we can show that there exists a \textit{canonical} ladder
associated with a given hyperbolic element in $\PSL(2,\mathbb{Z})$.
\begin{PROP}
  \label{prop:BiInfiniteLadder}
  For each hyperbolic element $f$, there exists the unique bi-infinite ladder $\mathcal{L}$
  stabilized by $f$. In this case, each side of $\mathcal{L}$ is stabilized by $f$ as well.
\end{PROP}

\begin{proof}
  Let $\mathcal{L}$ be the ladder associated with the axis $A_f$. Such a
  ladder is unique since $A_f$ is uniquely determined by $f$. Suppose
  for the sake of contradiction, $A_f$ intersects only finitely many
  Farey triangles. This is equivalent to saying $\mathcal{L}$ consists
  of finitely many Farey triangles. Then the ends of $A_f$ must be
  vertices of $\mathcal{F}$.  However, as two fixed points of $f$ on
  $A_f$ are irrational, the ends of $A_f$ cannot be vertices of
  $\mathcal{F}$, which is a contradiction. Hence, $\mathcal{L}$ must be
  bi-infinite.

  By definition of a ladder, $\mathcal{L}$ covers $A_f$.  As $f$
  bijectively sends a segment of $A_f$ to another segment of $A_f$,
  any subladder of $\mathcal{L}$ must be bijectively mapped to another
  subladder of $\mathcal{L}$.  Thus, it follows that $f$ stabilizes
  $\mathcal{L}$.

  \begin{figure}[h]\centering    
  \includegraphics[width = .7\linewidth]{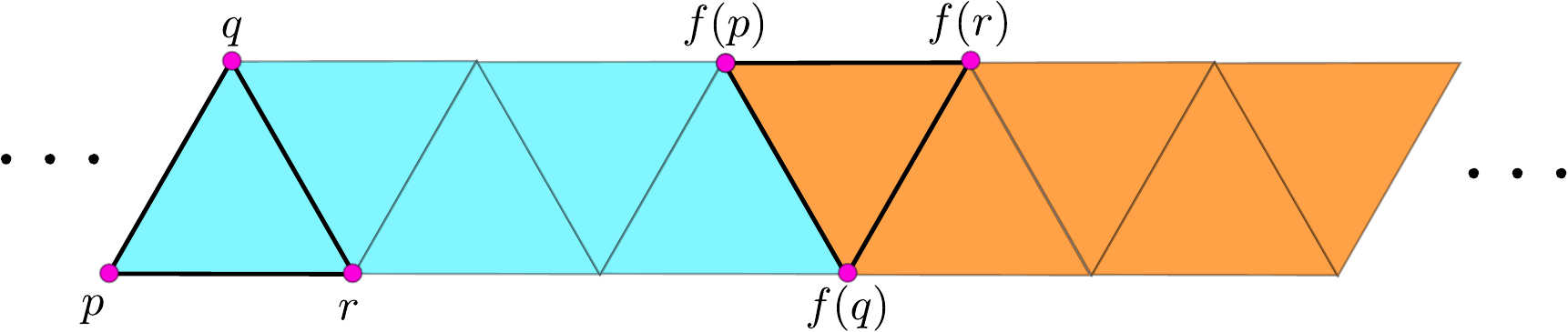}
  \caption[Odd length ladder]{$\triangle pqr$,
    $\triangle f(p)f(q)f(r)$: Two triangles with different orientations}
  \label{fig:odd_ladder}
\end{figure}

  To show $f$ stabilizes each side of $\mathcal{L}$, it suffices to
  show $f$ sends a pivot point of $\mathcal{L}$ to another pivot point
  on the same side of $\mathcal{L}$. Assume there is some pivot point
  $p$ in $\mathcal{L}$ such that $f(p)$ lies on the other side of
  $\mathcal{L}$, as in Figure \ref{fig:odd_ladder}.  Denote by $q,r$
  the adjacent pivot points of $p$ on the other side and the same side
  respectively. Since $p,q,r$ form a triangle, so do $f(p),f(q),f(r)$.
  Then as $f$ acts on $\mathcal{L}'$ by translation, the position of
  $f(p),f(q)$ and $f(r)$ must be as in Figure \ref{fig:odd_ladder}.
  However, then $\triangle pqr$ and $\triangle f(p)f(q)f(r)$ have
  different orientations, which is impossible since $f$ is orientation
  preserving. Therefore, $f$ must stabilize each side of
  $\mathcal{L}$.
\end{proof}

Now we can show that the canonical ladder associated with given
hyperbolic element must be \textit{periodic}.
\begin{THM}
  \label{thm:periodicBiInfiniteLadder}
  Let $f$ be a hyperbolic element in $\PSL(2,\mathbb{Z})$, and
  $\mathcal{L}$ be the bi-infinite ladder stabilized by $f$.  Then
  $\mathcal{L}$ is periodic with the period identical to the cutting
  sequence of a fixed point of $f$. More precisely, there exists a
  ladder $\mathcal{L}'$ constructed by a finite concatenation of a
  prime subladder of $\mathcal{L}$ such that
  \[
    \infcup f^k(\mathcal{L}') = \mathcal{L}.
  \]
\end{THM}

\begin{proof}
  Let $\alpha$ be a quadratic irrational fixed point of $f$.  By
  Proposition \ref{prop:CuttingAndPCF} and Corollary
  \ref{cor:CuttingAndNCF}, $\alpha$ has a periodic cutting sequence.
  Orient $A_f$ to make $\alpha$ be the terminal point of $A_f$.  By
  Lemma \ref{lem:cuttingSeqCoincides}, the cutting sequence of $A_f$
  must be eventually periodic. By Proposition
  \ref{prop:cuttingSeqAndLadder}, this induces the one-sided infinite
  periodic subladder in $ \mathcal{L}$ corresponding to the periodic part of
  the cutting sequence of $A_f$.

\begin{figure}[h]\centering    
  \includegraphics[width = .7\linewidth]{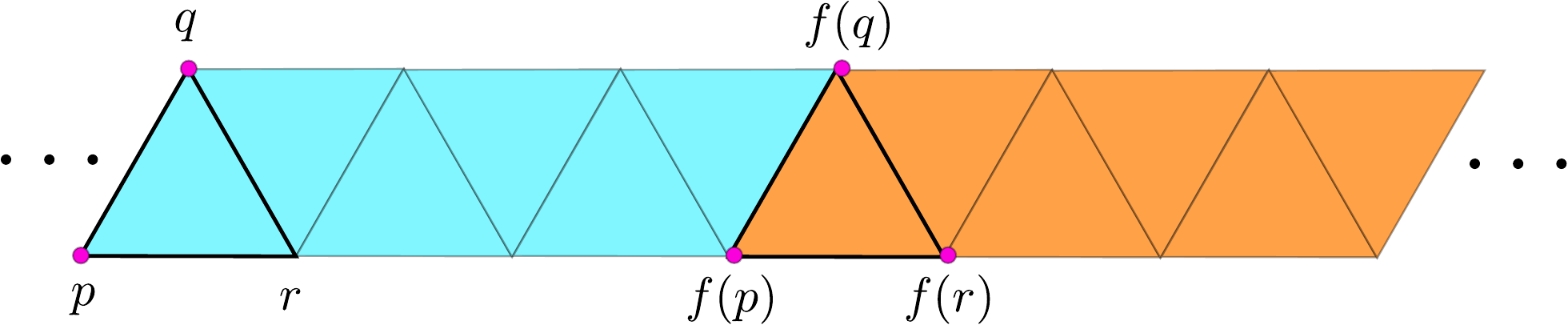}
  \caption[Even length ladder]{Periodic part of bi-infinite ladder $\mathcal{L}$.}
  \label{fig:even_ladder}
\end{figure}
  
Let $\mathcal{L}''$ be its prime subladder, with orientation inherited
from $A_f$, and $p,q$ be its first two pivot points.
As $f(p),f(q)$ are also pivot points of $\mathcal{L}$, Now
define $\mathcal{L}'$ to be the finite subladder of $\mathcal{L}$
bounded by two edges $\overline{pq}$ and $\overline{f(p)f(q)}$. Since
$f$ acts on $\mathcal{L}$ by translation, we have
  \[
    \infcup f^k(\mathcal{L}') = \mathcal{L},
  \]
  and thus $\mathcal{L}$ is periodic, with the period identical to the
  period of $\mathcal{L}'$, which is same as the cutting sequence of
  $\alpha$.

  Now, suppose $\mathcal{L}'$ cannot be formed by a finite
  concatenation of the prime subladder $\mathcal{L}''$ of
  $\mathcal{L}$.  Since both $\mathcal{L}'$ and $\mathcal{L}''$
  bi-infinitely cover $\mathcal{L}$ having the same starting pivot
  points as $p$ and $q$, the length of the period of $\mathcal{L}$
  must be a common divisor $d$ of the two lengths of
  periods of $\mathcal{L}'$ and $\mathcal{L}''$. Since we assumed
  $\mathcal{L}'$ cannot be formed by $\mathcal{L}''$, $d$ must be
  strictly less than the length of $\mathcal{L}''$.  However, from the
  minimality of prime subladder, the value $d$ must be odd. This implies
  the period of $\mathcal{L}'$ is odd, since the length of a prime
  subladder is always even. Thus $p$ and $f(p)$ must lie on the different
  side of $\mathcal{L}$, which is absurd by the second statement of
  Proposition \ref{prop:BiInfiniteLadder}.
\end{proof}

\section{Translation Length on Farey Graph}
\label{sec:TrLengthOnFareyGraph}
In this section, we provide our main theorem: there exists a
\textit{geodesic axis} in the Farey graph of a given Anosov mapping class.
This theorem allows to define the \textit{translation length} of
Anosov mapping class on the Farey graph, which is our initial objective to
calculate. Indeed, after proving the main theorem, we provide a
concrete algorithm to calculate the translation length of any
Anosov element.

\begin{THM}[Existence of Geodesic Axis]
  \label{thm:geodAxis}
  Let $f$ be an Anosov map. Then there exists a bi-infinite geodesic
  $\mathcal{P}$ in $\mathcal{F}$ on which $f$ acts by translation.
\end{THM}

\begin{proof}
  We identify $f$ as a hyperbolic element of $\PSL(2,\mathbb{Z})$. Then the
  existence of the invariant bi-infinite geodesic $\mathcal{P}$ of $f$
  is a direct consequence of Proposition \ref{prop:ConcatGeodesic}
  and Theorem \ref{thm:periodicBiInfiniteLadder}. Indeed, there
  exists a ladder $\mathcal{L}'$ which is a finite concatenation of a
  prime subladder of the associated bi-infinite ladder $\mathcal{L}$
  of $f$ satisfying:
  \[
    \infcup f^k(\mathcal{L}') = \mathcal{L}.
  \]
  Hence, there exists an efficient geodesic $\mathcal{P}'$ in $\mathcal{L}'$ whose bi-infinite
  concatenation makes the efficient geodesic $\mathcal{P}$ in $\mathcal{L}$.
  Therefore, $\mathcal{P}$ is invariant under $f$, as
  \[
    \infcup f^k(\mathcal{P}') = \mathcal{P}.
    \]
  Since $f$ acts on $\mathcal{F}$ by translation, so does on $\mathcal{P}$, which was what was wanted.
\end{proof}

\begin{RMK}
  Since we have considered only \textit{efficient} bi-infinite
  geodesics, there may be another geodesic in $\mathcal{L}'$, whose
  concatenation makes another $f$-invariant bi-infinite geodesic in
  $\mathcal{F}$. In particular, the geodesic axis $\mathcal{P}$ in
  $\mathcal{F}$ of $f$ in may \textit{not} be unique.
\end{RMK}

Now we provide a concrete algorithm to calculate the translation
length of an Anosov mapping class. To this end, we introduce a special
map called \textit{ancestor map} on vertices of the Farey graph, which can
be found in \cite{beardon2012geodesic}.  Let $\VERT(\mathcal{F})$ be
the set of vertices of the Farey graph.  For each
$v \in \VERT(\mathcal{F})$, we can give a lexicographical order on the
neighbors of $v$, first order by denominator in increasing order, then
by numerator as the same way. Since we restrict the denominators of
elements of $\VERT(\mathcal{F})$ to be positive, this ordering on
$\VERT(\mathcal{F})$ is a well-ordering. Therefore, the map
$\alpha : \VERT{\mathcal{F}} \to \VERT{\mathcal{F}}$ associating each
vertex to its smallest(with respect to the order we gave) neighbor is
well-defined. We call this map $\alpha$ as \textbf{ancestor map}.
Note that for any extended rational number $p$, we have
$\alpha^n(p) = \infty$ for some non-negative $n$.  Choosing such $n$
as the minimum one, we have a path from $p$ to $\infty$ obtained by iterating
ancestor map on $p$:
\[
  p, \alpha(p), \alpha^2(p), \cdots, \alpha^n(p)=\infty.
\]
We call such a path \textbf{ancestor path}, which turns out to be a
\textit{geodesic} in $\mathcal{F}$ connecting $p$ and
$\infty$.(See \cite{beardon2012geodesic})

When two fixed points of a hyperbolic element have opposite signs,
then the bi-infinite geodesic connecting these two points and the edge
$e = \overline{0,\infty}$ must intersect. Thus, in this case, $e$ must be a rung of
the associated ladder of $f$. We call such a ladder containing
$e=\overline{0,\infty}$ a \textbf{standard ladder}. In fact, whether a given
hyperbolic element has a standard invariant ladder or not can be
immediately identified by their signs of entries.

\begin{PROP}
  \label{prop:stdrung}
  Let $A =
  \begin{pmatrix}
    a & b \\ c & d
  \end{pmatrix}
  \in \PSL(2,\mathbb{Z})$ be a hyperbolic element. Then the associated
  ladder of $A$ is standard if and only if $b, c$ have the
  same sign.
\end{PROP}

\begin{proof}
  Let $x \in \mathbb{R}$ be one of the fixed points of $A$. Then we
  have $\frac{ax+b}{cx+d} = x$, so $cx^2+(d-a)x-b=0$. Since the
  associated ladder of $A$ contains $e=\overline{0,\infty}$ if and
  only if the hyperbolic axis of $A$ transverses $e$ if and only if two fixed
  points have different sign if and only if $\frac{-b}{c}<0$ if and
  only if $b$ and $c$ have the same sign.
\end{proof}

The following proposition suggests a hyperbolic element with a standard invariant ladder
is quite \textit{standard within its conjugacy class} of $\PSL(2,\mathbb{Z})$. This result
will be needed in section \ref{ssec:Lengthspectrum}.

\begin{PROP}
  \label{prop:stdconj}
  Any hyperbolic element of $\PSL(2,\mathbb{Z})$ is conjugate to an element with standard
  invariant ladder.
\end{PROP}

\begin{proof}
  Let $f$ be a hyperbolic element of $\PSL(2,\mathbb{Z})$ with an invariant ladder $\mathcal{L}$.
  Pick any rung $\overline{\frac{p}{q},\frac{r}{s}}$ of $\mathcal{L}$. Then one of the following two matrices
  \[
    \begin{pmatrix}
      p & r \\ q & s
    \end{pmatrix}
    \quad \text{or}
    \quad
    \begin{pmatrix}
      r & p \\ s & q
    \end{pmatrix}
  \]
  must be in $\PSL(2,\mathbb{Z})$. Denote by $A$ the one in
  $\PSL(2,\mathbb{Z})$. In fact, $A$ maps the rung
  $\overline{\frac{p}{q},\frac{r}{s}}$ of $\mathcal{L}$ to an edge
  $\overline{0,\infty}$, so $\mathcal{L}$ is mapped to a standard ladder
  $\mathcal{L}'$ under $A$. In fact, the very hyperbolic element $A f A^{-1} \in \PSL(2,\mathbb{Z})$ 
  is associated with the standard invariant ladder $\mathcal{L}'$, which was what was
  wanted.
\end{proof}

Aside from standard ladders, we can find a rung of the associated ladder using ancestor map:

\begin{PROP}
  \label{prop:paraladder}
  Let $f$ be a hyperbolic element in $\PSL(2,\mathbb{Z})$. Let $\delta$ and $\overline{\delta}$
  be the two fixed points of $f$ and set $m = \frac{\delta + \overline{\delta}}{2}$.
  Then there exists $k \ge 0$ such that the edge formed by $\alpha^k(m)$ and $\alpha^{k+1}(m)$
  is a rung of the associated bi-infinite ladder of $f$.
\end{PROP}

\begin{proof}
  Without loss of generality, let $\overline{\delta} < \delta$.  Note
  that $m$ is a rational number, so $m$ represents a vertex in
  $\mathcal{F}$.  Also, $m$ and $\infty$ are separated by the
  bi-infinite geodesic with ends $\delta$ and $\overline{\delta}$. As
  the ancestor path for $m$ is a path from $m$ to $\infty$, there
  exists $k \ge 0$ such that
  $\alpha^k(m) \in (\overline{\delta},\delta)$ and
  $\alpha^{k+1}(m) \notin (\overline{\delta}, \delta)$.  Then the edge
  $e$ with ends $\alpha^{k}(m)$ and $\alpha^{k+1}(m)$ must transverse
  the geodesic connecting $\overline{\delta}$ and $\delta$. By
  Proposition \ref{prop:E-EdgesAreRungs}, the edge $e$ must be a rung
  of the bi-infinite ladder associated with $f$.
\end{proof}

Now, write $p=\alpha^k(m)$. Then denote by $\widetilde{\mathcal{L}}$
the subladder bounded by two edges $\overline{p, \alpha(p)}$ and
$\overline{f(p), f(\alpha(p))}$. If the length of
$\widetilde{\mathcal{L}}$ is \textit{even}, then $p$ must be a pivot
point, so $\widetilde{\mathcal{L}}$ comes out to be a finite
concatenation of a prime subladder of
$\mathcal{L}(\overline{\delta},\delta)$. In turn, we are left to
realize the efficient geodesic in $\widetilde{\mathcal{L}}$ and
calculate its length to find the translation length of $f$.  However,
if the length of $\widetilde{L}$ is \textit{odd}, then $p$ is not a
pivot point of $\mathcal{L}$. In this case, we need \textit{calibrate}
$\widetilde{L}$ to find another subladder with an even length. The calibration
process is simple: From an odd-length ladder of type
$(a_1,\cdots,a_n)$, generate an even-length ladder
$(a_1+a_n,a_2,\cdots,a_{n-1})$ whose bi-infinite concatenation with
$f$ yields $\mathcal{L}$ as well. (Refer to Algorithm
\ref{alg:calLadder} in Appendix) Thus, the case when $\widetilde{\mathcal{L}}$ has an odd length
is reduced to when $\widetilde{\mathcal{L}}$ has an even length.

It remains to realize an efficient geodesic in the \textit{even}
length subladder $\widetilde{\mathcal{L}}$ to find the translation
length of $f$.  Since the types of the ladder are given, we can use
the criteria introduced in Section \ref{sec:EffGeodesicsInLadder} to
find the efficient geodesic. To illustrate the aforementioned
process, we provide some examples:

\begin{EXA}
  \label{exa:exa1}
  Consider an Anosov map $f$ represented by a matrix in $\PSL(2,\mathbb{Z})$:
  \[
    \begin{pmatrix}277 & 60 \\ 337 & 73\end{pmatrix}
  \]
  The fixed points of $f$ are $\frac{77 \pm \sqrt{26149}}{337}$. Since
  their signs are different.  the ladder associated with $f$ must be
  standard.  In fact, the ladder bounded by two edges
  $e=\overline{0,\infty}$ and $f(e)$ is of the type
  $(1,4,1,1,1,1,1,1,4)$, which can be calibrated into one with
  $(4,1,1,1,1,1,1,5)$. (See Figure \ref{fig:exa1}) Thus, the
  corresponding efficient geodesic is $tpppt$, length $5$.
  Hence, the translation length of $f$ is $5$.
  \begin{figure}[h]\centering    
  \includegraphics[width = \linewidth]{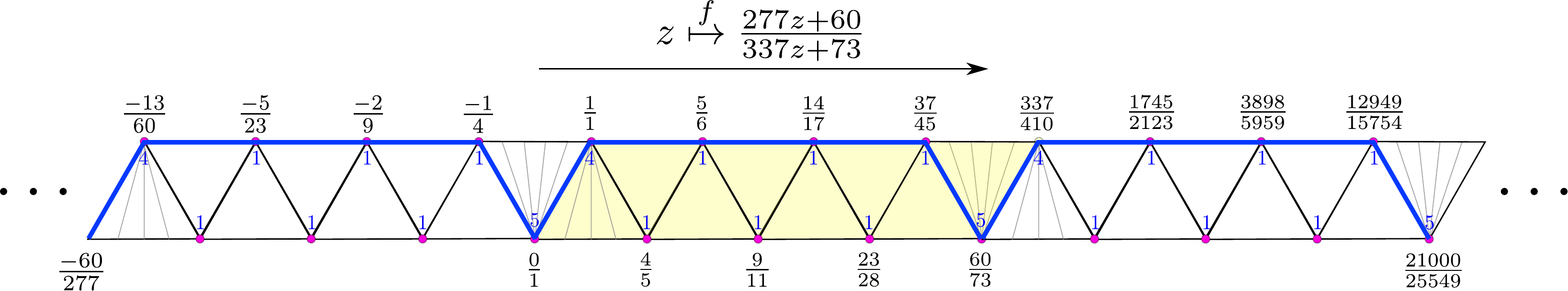}
  \caption[]{Invariant geodesic for $f$ represented by $\begin{pmatrix}227 & 60 \\ 337 & 73 \end{pmatrix}$}
  \label{fig:exa1}
\end{figure}
\end{EXA}

Now we provide a non-standard ladder example which forces to use Proposition \ref{prop:paraladder}
to find a rung of the associated ladder.

\begin{EXA}
  \label{exa:exa2}
  Consider an Anosov map $f$ represented by a matrix in $\PSL(2,\mathbb{Z})$:
  \[
    \begin{pmatrix}65 & -56 \\ 101 & -87\end{pmatrix}
  \]
\end{EXA}
In this case, two fixed points of $f$ are
$\frac{76 \pm 2\sqrt{30}}{101}$, so its midpoint $m
=\frac{76}{101}$. Since two points have the same sign, we exploit Proposition
\ref{prop:paraladder} to find a rung in the associated ladder
$\mathcal{L}$.  The ancestor path for $\frac{76}{101}$ is:
\[
  \frac{76}{101} \to \frac{3}{4} \to \frac{1}{1} \to \frac{1}{0}.
\]
Since two fixed points are approximately $0.644\cdots$ and
$0.860\cdots$, the edge $e$ connecting $\frac{3}{4}$ and $\frac{1}{1}$
must be a rung of $\mathcal{L}$.  Then the ladder bounded by
$e=\overline{\frac11,\frac34}$ and $f(e)$ is of the type $(1,5,3)$, which can be
calibrated into one with $(5,4)$, in which an efficient geodesic has
the form $tt$. (see Figure \ref{fig:exa2}) Hence, the translation
length of $f$ is $2$.
\begin{figure}[ht]\centering    
  \includegraphics[width = \linewidth]{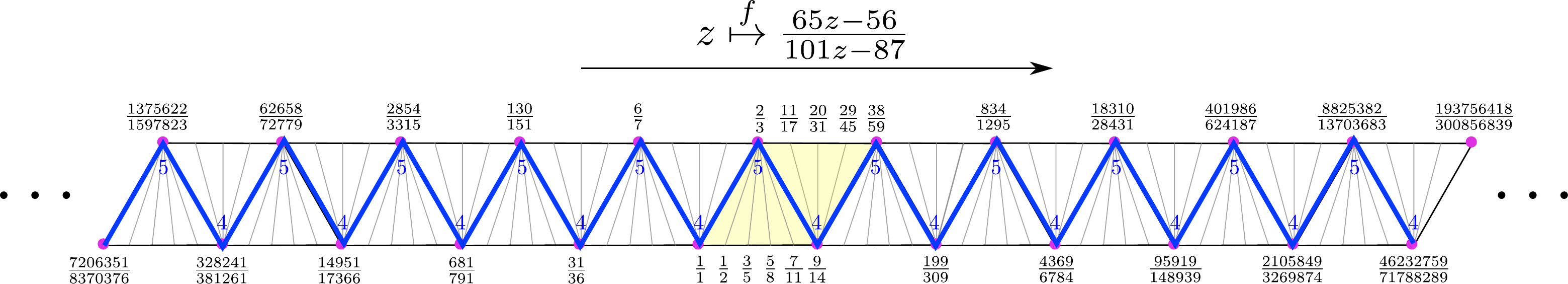}
  \caption[]{Invariant geodesic for $f$ represented by $\begin{pmatrix}65&-56\\101&-87\end{pmatrix}$}
  \label{fig:exa2}
\end{figure}

\section{Applications}
\label{sec:appl}

\subsection{Minimal Word}
\label{ssec:minword}

Shin and Strenner \cite{shin2016pseudo} showed that there
is a non-Penner type pseudo-Anosov mapping class $\Mod(S)$ when $S$ is a non-sporadic surface.
They also showed this is not the case when $S$ is a sporadic surface.
In particular, they showed the following(Lemma A.1 in \cite{shin2016pseudo}):

\begin{LEM}
  Let $f$ be a hyperbolic element in $\mathcal{F}$, and $\gamma$ be
  the associated bi-infinite geodesic.  Let $e_0$ be an edge of
  $\mathcal{F}$ intersecting $\gamma$.  Let $a$ and $b$ be the
  endpoints of $e_0$ on the left and right hand side of $\gamma$,
  respectively. Following the orientation of $\gamma$, record the cutting
  sequence $s_1\cdots s_n$ starting from the next rung of $e_0$ and ending at $e_n =f(e_0)$.
  Let $\tau_a$ and $\tau_b$ be the rotations of
  $\mathcal{F}$ by one tile to the right about the points $a$ and $b$,
  respectively, and introduce the notation
  \[
    \tau(s)=
    \begin{cases}
      \tau_a^{-1} & \text{if $s = L$}\\
      \tau_b & \text{if $s=R$}.
      \end{cases}
    \]
    Then
    \[
      f = \tau(s_1) \circ \cdots \circ \tau(s_n),
    \]
    where the rotations are applied in right-to-left order.
\end{LEM}

Thus, we can regard a hyperbolic element of $\PSL(2,\mathbb{Z})$ as a
word with letters $\{\tau_a^{-1}, \tau_b\}$, where $a$ and $b$ are
chosen to form an edge $\overline{ab}$ intersecting $\gamma$, so
$\overline{ab}$ becomes a rung of the associated ladder to $f$.

Then the natural question to ask is the following: \textit{Using the fixed number of
  $\tau_a^{-1}$ and $\tau_b$, what is the form of words whose
  translation length is minimal?}
Our result answers to this question.

\begin{THM}
  \label{thm:minimal_length}
  Let $a$ and $b$ be isotopy classes of essential simple closed
  curves with intersection number 1 on a torus.
  Let
  \[
    \mathcal{S} = \{\underbrace{\tau_a, \cdots, \tau_a}_m,
    \underbrace{\tau_b, \cdots, \tau_b}_n\},
  \]
  where $m,n \ge 1$.
  Then among the words formed by the elements in $\mathcal{S}$,
  the word $\tau_a^{-m}\tau_b^n$(upto cyclic permutation) has the shortest translation length.
  Moreover, that shortest length is at most 2.
\end{THM}

\begin{proof}
  The latter assertion obvious, since the length of the ladder
  associated with $\tau_a^{-m}\tau_b^n$ is 2.
  
  When at least one of $m$ or $n$ is $1$, there is only one word
  formed by the elements in $\mathcal{S}$ upto cyclic permutation, so
  there is nothing to prove.

  Now assume both $m$ and $n$ are bigger than 1. Then a word $w$ which
  is not of the form $\tau_a^{-m}\tau_b^n$ upto cyclic permutation
  must have the ladder with length at least 4. However, in a ladder
  with length 4, the length of the efficient geodesic is at least
  2(`$pp$', if possible, is the shortest one), so $w$ cannot have
  the shorter translation length than $\tau_a^{-m}\tau_b^n$ have.
  Therefore, we proved $\tau_a^{-m}\tau_b^n$ produces the smallest
  translation length on $\mathcal{F}$ among the words formed by $\mathcal{S}$.
\end{proof}

\subsection{Effective Bound of Ratio of Teichm{\" u}ller to Curve Graph Translation Length}
\label{ssec:Lengthspectrum}

We can define the translation length on \textit{Teichm{\" u}ller
  space} as well. More precisely, for a non-sporadic surface
$S$ and any pseudo-Anosov map
$f \in \Mod(S)$, Bers \cite[Theorem 5]{bers1978extremal} showed that
there exists the axis $A_f$ in the Teichm{\" u}ller space
$\mathcal{T}(S)$ of $S$ which is invariant under the action of $f$ by
translation with the Teichm{\" u}ller distance
$l_{\mathcal{T}}(f) := \log \lambda$, where $\lambda>1$ is the
dilatation of $f$. The analogous property also holds for when $S$ is a
torus: an Anosov map acts on the Teichm{\" u}ller space of a torus
$\mathcal{T}(T^2)$ with the invariant axis by translation with the
hyperbolic distance $\log \lambda$, where $\lambda>1$ is the
leading eigenvalue of $f$.

 For non-sporadic surfaces $S=S_g$, Gadre, Hironaka, Kent and Leininger
\cite{Gadre2013Lipschitz} showed that the ratio
$\frac{l_{\mathcal{T}}(f)}{l_{\mathcal{C}}(f)}$ is bounded below by a
linear function of $\log(g)$ by constructing a pseudo-Anosov map $f$
which realizes the minimum. Aougab and Taylor \cite{aougab2015pseudo}
provided another infinite family of such ratio minimizers. We show
here the analogous result on the ratio for a torus also holds.

  To begin with, we formalize our observations on the connection between continued fractions and the ladders. For a finite sequence of positive integers $\mathcal{S}=(a_1,a_2,\cdots,a_n)$,
  we consider the following two continued fractions: when $n>1$, write
  \[
    \frac{p(\mathcal{S})}{q(\mathcal{S})} := [a_1;a_2,\cdots,a_n],
    \qquad \frac{r(\mathcal{S})}{s(\mathcal{S})} := [a_1;a_2,\cdots,a_{n-1}]
  \]
  where $(p(\mathcal{S}),q(\mathcal{S}))$ and
  $(r(\mathcal{S}),s(\mathcal{S}))$ are relatively prime positive
  integer pairs for $n>1$.  When $n=1$, define
  $(p(\mathcal{S}), q(\mathcal{S}), r(\mathcal{S}),
  s(\mathcal{S}))=(a_1,1,1,0)$ to make it consistent with the case of
  $n>1$.

  Observe that two fractions $\frac{p(\mathcal{S})}{q(\mathcal{S})}$
  and $\frac{r(\mathcal{S})}{s(\mathcal{S})}$ are realized by two
  adjacent pivot points forming an edge $e$ of a standard ladder of
  type $\mathcal{S}$, which is flanked by $e_0=\overline{0,\infty}$ and $e$.
  We denote by $\mathcal{L}(\mathcal{S})$ this
  standard ladder constructed from a sequence $\mathcal{S}$. Now, a
  matrix defined as
  \[
    \mathcal{M}(\mathcal{S}) :=
  \begin{pmatrix}
    p(S) & r(S) \\ q(S) & s(S)
  \end{pmatrix}
  \in \GL(2,\mathbb{Z})
\]
sends $e_0$ to $e$, so $\mathcal{M}(\mathcal{S})$ has the associated
ladder as $\mathcal{L}(\mathcal{S})$. Indeed, the function $\mathcal{M}$
from the set of finite sequences of positive integers to
$\GL(2,\mathbb{Z})$ is a \textit{structure-preserving map}: The matrix
constructed from the concatenation of two positive integer sequences can
be expressed by the product of each associated matrix.

\begin{LEM}
  \label{lem:concatM}
  Let $(a_1,\cdots, a_m)$ and $(b_1,\cdots,b_n)$ be two sequences
  of positive integers. Then
  \[
    \mathcal{M}(a_1,\cdots,a_m)\mathcal{M}(b_1,\cdots, b_n) =
    \mathcal{M}(a_1,\cdots, a_m, b_1, \cdots, b_n).
  \]
\end{LEM}
\begin{proof}
  We claim
  \[
    \mathcal{M}(a_1,\cdots,a_n) = \mathcal{M}(a_1)\mathcal{M}(a_2,\cdots,a_n)
  \]
  for every $n$. Write $\mathcal{M}(a_2,\cdots,a_n) =
  \begin{pmatrix}
    p' & r' \\ q' & s'
  \end{pmatrix}
$. Then
\[
  \mathcal{M}(a_1)\mathcal{M}(a_2,\cdots,a_n)=
    \begin{pmatrix}
      a_1 & 1 \\ 1 & 0
    \end{pmatrix}  
    \begin{pmatrix}
      p' & r' \\ q' & s'
    \end{pmatrix}
    =
    \begin{pmatrix}
      a_1p' + q' & a_1 r' + s' \\ p' & r'
    \end{pmatrix}.
\]
  By definition,
  \begin{align*}
    \frac{a_1p'+q'}{p'} &= a_1 + \frac{1}{p'/q'}= a_1 + [0;a_2,\cdots,a_n] = [a_1;a_2,\cdots,a_n], \quad \\
    \frac{a_1r'+s'}{r'} &= a_1 + \frac{1}{r'/s'}= a_1 + [0;a_2,\cdots,a_{n-1}] = [a_1;a_2,\cdots,a_{n-1}].
  \end{align*}
  Hence, we have
  $\begin{pmatrix}
      a_1p' + q' & a_1 r' + s' \\ p' & r'
    \end{pmatrix} = \mathcal{M}(a_1,\cdots,a_n)$, as we have claimed.

  Finally, the claim implies Lemma \ref{lem:concatM}. This is because:
  \begin{align*}
    \mathcal{M}(a_1,\cdots,a_m)\mathcal{M}(b_1,\cdots, b_n) &= \mathcal{M}(a_1) \cdots \mathcal{M}(a_m) \mathcal{M}(b_1) \cdots \mathcal{M}(b_n) \\ &= \mathcal{M}(a_1,\cdots,a_m,b_1,\cdots,b_n).
  \end{align*}
  \end{proof}

  To extract an inequality of translation lengths from that of
  matrices, we introduce the component-wise comparison among matrices.
  \begin{DEF}
    For two matrices $A=(a_{ij})$ and $B=(b_{ij})$ having the same
    size, we write $A \succeq B$ if $a_{ij} \ge b_{ij}$ for all
    $i,j$, that is, every component of $A$ is greater than or equal to
    that of $B$. In particular, we will write it simply $A \succeq 0$ when
    every component of $A$ is nonnegative.
  \end{DEF}

  It is easily seen that for equisized square matrices $A,B,C,D$, and
  $A \succeq B \succeq 0$ and $C \succeq D \succeq 0$, we have
  $AC \succeq BD \succeq 0$. Hence by Lemma \ref{lem:concatM}, two
  positive integer sequences $(a_1,\cdots, a_n)$ and $(b_1,\cdots, b_n)$
  with $a_i \ge b_i$ for every $i$ induces the inequality
  $\mathcal{M}(a_1,\cdots,a_n) \succeq \mathcal{M}(b_1,\cdots,b_n)$. The same argument
  proves $\mathcal{M}(a_1,\cdots,a_n) \succeq 0$, as $\mathcal{M}(a_i) \succeq 0$ for every $i$.
  We summarize the observations in the following lemma.
\begin{LEM}
  \label{lem:MatrixIneq}
  For two positive integer sequences $(a_1,\cdots, a_n)$ and
  $(b_1,\cdots, b_n)$ with $a_i \ge b_i$ for every $i$, we have
  \[
    \mathcal{M}(a_1,\cdots,a_n)\succeq \mathcal{M}(b_1,\cdots,b_n)
    \succeq 0.
  \]
\end{LEM}

Shortly, we will see that how big is the \textit{trace} of an Anosov
map $f$ is closely related to the translation length of $f$. 
Since we are identifying $f$ as a hyperbolic element in
$\PSL(2,\mathbb{Z})$, we have to first clarify what it means by the \textit{trace}
of an element in $\PSL(2,\mathbb{Z})$. Indeed, due to the ambiguity of
signs in $\PSL(2,\mathbb{Z})$, the trace $\tr : A \in \PSL(2,\mathbb{Z}) \mapsto \tr A$ is not
well-defined.  Thus, we use the absolute value $|\tr A|$ to define the
trace on $\PSL(2,\mathbb{Z})$, which resolves the ambiguity. We write
$\tr A:=|\tr A|$ when no confusion can arise.

Note that $\det\mathcal{M}(a_1)=-1$, so we have
$\det \mathcal{M}(a_1,\cdots,a_k) = (-1)^k$ by Lemma \ref{lem:concatM}. Thus
whenever a sequence $\mathcal{S}$ has an \textit{even} length, $\mathcal{M}(S)$
represents an element of $\PSL(2,\mathbb{Z})$, whose translation length actually
decides the lower bound of the trace:

\begin{PROP}
  \label{prop:StandardIneq}
  Let $(a_1,\cdots,a_{2n})$ be a sequence of positive integers.
  Then
  \begin{align}
    \label{eqn:StandardIneq}
    \tr \mathcal{M}(a_1,\cdots, a_{2n} ) \ge \tr(\mathcal{M}(2)^l),
  \end{align}
  where $l$ is the translation length of $\mathcal{M}(a_1,\cdots,a_{2n})$ on $\mathcal{F}$.
\end{PROP}

\begin{proof}
  We first consider the case $a_1=a_2=\cdots=a_{2n}=1$. Since
  $\mathcal{M}(1,\cdots,1)$ has the invariant geodesic of the form
  $p \cdots p$ in $\mathcal{F}$, the translation length of
  $\mathcal{M}(\underbrace{1,\cdots,1}_{2n})$ is n. By Lemma \ref{lem:concatM},
  \vspace{-1em} 
  \[
    \mathcal{M}(a_1,\cdots, a_{2n})=\mathcal{M}(1)^{2n}
    =\begin{pmatrix} 1 & 1 \\ 1 & 0 \end{pmatrix}^{2n}
    =\begin{pmatrix} 2 & 1 \\ 1 & 1 \end{pmatrix}^{n}
    \succeq \begin{pmatrix} 2 & 1 \\ 1 & 0 \end{pmatrix}^{n}
    = M(2)^n.
\]
Applying the trace function on both sides, we have the inequality \eqref{eqn:StandardIneq}.

Now suppose $a_i>1$ for some $i$.
Note that $\mathcal{M}(a_i,\cdots,a_{2n}, a_1,\cdots, a_{i-1})$ is conjugate
to $\mathcal{M}(a_1,\cdots,a_{2n})$ for every $i$ in $\PSL(2,\mathbb{Z})$ by Lemma
\ref{lem:concatM}. Since trace and translation length are invariant under
conjugation, now we may assume $a_1 >1$.

The rest of the proof will be divided into two steps.  First, we will
\textit{downsize} $\mathcal{M}(a_1,\cdots,a_{2n})$ by constructing a
smaller sequence $(b_1,\cdots,b_{2n})$ out of $(a_1,\cdots,a_{2n})$
while $\mathcal{M}(b_1,\cdots,b_{2n})$ has the same translation
length. Then we will show $\mathcal{M}(b_1,\cdots,b_{2n})$ satisfies
the inequality \eqref{eqn:StandardIneq} with $a_i$'s replaced by
$b_i$'s.

\textbf{Step 1(Construction of Downsized Sequence.)}
Based on the sequence $(a_1,\cdots,a_{2n})$, we construct a sequence $(b_1,\cdots, b_{2n})$ as follows:
\begin{enumerate}[(i)]
\item Set $b_i=1$ whenever $a_i=1$.
\item For each odd block of $1$'s enclosed by non-$1$'s in $\{a_i\}$,
  set the corresponding element of $\{b_i\}$ in the position of non-1
  followed by the odd 1-block to be 1. More precisely, whenever there exist
  $i,k$ such that $a_{i}>1, a_{i+1}=a_{i+2}=\cdots =a_{i+2k-1}=1$, and
  $a_{i+2k}>1$, then let $b_{i+2k}=1$.
\item Set all undetermined $b_i$ to $2$.
\end{enumerate}
\begin{figure}[ht]\centering
  \includegraphics[width = \textwidth]{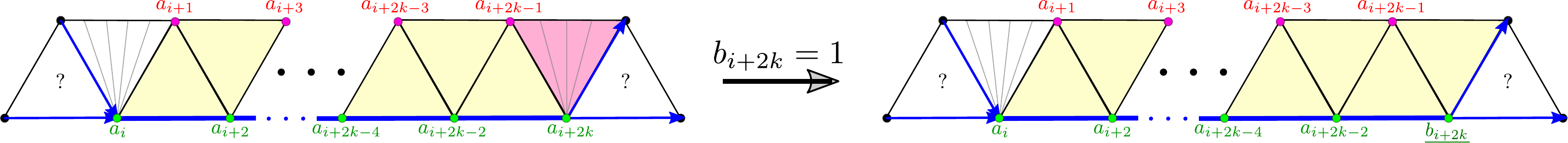}
  \caption{Description of Process $(ii)$. Note the efficient
    geodesic(Blue lines) does not change under the process. In both
    cases, the efficient geodesic never visits red vertices, but
    always does green ones.}
\label{fig:downsizing}
\end{figure}

Recall that the efficient geodesic in a ladder solely depends on the type
of the ladder.  Thus, obviously the process (i) has no effect on
deciding efficient geodesics in the ladder.  For (iii), note that the
move $t$ is done whenever the move $p$ visits $\ge 2$ vertices, so
changing a non-1 coefficient $a_i \ge 2$ into $b_i=2$ does not change
the efficient geodesic as well. Therefore, what is left to show is
that the transition of non-1 coefficient $a_{i+2k}$ into 1 in Process (ii)
does not affect deciding the efficient geodesic in the ladder. Suppose
$a_i>1$, $a_{i+1}=\cdots=a_{i+2k-1}=1$, and $a_{i+2k}>1$. Since
$a_i>1$, the efficient geodesic must visit the pivot point
corresponding to $a_i$, followed by $k$-consecutive $p$ moves by which
it avoids the pivot points associated with
$a_{i+1},a_{i+3},\cdots, a_{i+2k-1}$ until it reaches the point
associated with $a_{i+2k}$. In particular, the part of efficient
geodesic decided by $a_i,\cdots,a_{i+2k}$ is completely irrelevant
with the value of $a_{i+2k}$. See Figure
\ref{fig:downsizing}. Therefore we may set $b_{i+2k}=1$ without
changing the efficient geodesic of $\mathcal{M}(a_1,\cdots,a_n)$.

All in all, the three processes $(i),(ii),(iii)$ preserve the efficient geodesic
in the ladder, so $\mathcal{M}(a_1,\cdots, a_n)$ and
$\mathcal{M}(b_1,\cdots, b_n)$ have the same translation length on
$\mathcal{F}$, satisfying $a_i \ge b_i$, by construction.  By Lemma
\ref{lem:MatrixIneq}, we get
\begin{align}
  \label{eqn:CompIneq}
\tr \mathcal{M}(a_1,\cdots, a_n)\geq \tr \mathcal{M}(b_1,\cdots, b_n).
\end{align}

\textbf{Step 2 (Show Inequality with $\mathbf{b_i}$'s.)} We show the
similar inequality of \eqref{eqn:StandardIneq} for $b_i$ also holds:
  \begin{align}
    \label{eqn:VariantIneq}
    \tr \mathcal{M}(b_1,\cdots, b_{2n} ) \ge \tr(\mathcal{M}(2)^l).
  \end{align}
  By construction $(ii)$, any odd 1-blocks in $\{a_i\}$ are all
  transformed to even 1-blocks in $\{b_i\}$, except for those
  \textit{not} enclosed by non-1 blocks. However, as we have assumed
  $a_1>1$, an odd block of $1$ in $b_i$, if exists, can only appear at
  the end of the sequence. Thus, if the number of $1$ in $b_i$ is
  $2k$, then it forces every $1$-block in $\{b_i\}$ to be
  even. Therefore the efficient geodesic will contain $2n-2k$ of $t$-moves
and $k$ of $p$-moves, so the translation length of
  $\mathcal{M}(b_1,\cdots,b_{2n})$ must be $2n-k$. On the other hand,
  as $1$'s are clustered in even blocks, every $\mathcal{M}(1)$ in
  $\mathcal{M}(b_1)\mathcal{M}(b_2)\cdots \mathcal{M}(b_{2n})$ can be
  paired and reduced into $\mathcal{M}(2)$, for
  $\mathcal{M}(1)^2 \succeq \mathcal{M}(2)$.  Then we have
  \[
    \mathcal{M}(b_1)\mathcal{M}(b_2)\cdots \mathcal{M}(b_{2n})\succeq
    \mathcal{M}(2)^{2n-k},
  \]
  which establishes the inequality \eqref{eqn:VariantIneq}.

  If the number of $1$ in $b_i$ is $2k-1$, then it means there is an
  odd block of $1$'s at the end of $\{b_i\}$, and in particular
  $b_{2n}=1$. To exploit as many as $1$'s to minimize the length of
  the efficient geodesic, we have to start at $\infty$, which in turn
  yields a shorter geodesic with and $2n-2k$ of $t$-moves and $k$ of
  $p$-moves, compared to the geodesic starting from $0$, which
  consists of $2n-2k+2$ of $t$-moves and $k-1$ of $p$-moves. Thus, the
  translation length of $\mathcal{M}(b_1,\cdots,b_{2n})$ is
  $2n-k$. Since $\mathcal{M}(2)\mathcal{M}(1)\succeq \mathcal{M}(2)$,
  we can ignore the first $1$ appearing in the odd block of $1$'s. For the other
  $1$'s, we can pair them as before, which will give
  \[
    \mathcal{M}(b_1)\mathcal{M}(b_2)\cdots \mathcal{M}(b_{2n})\succeq
    \mathcal{M}(2)^{2n-k},
  \]
  again showing the inequality \eqref{eqn:VariantIneq}.

Finally, combining \eqref{eqn:CompIneq} with \eqref{eqn:VariantIneq}, we conclude
\[
  \tr \mathcal{M}(a_1,\cdots,a_{2n}) \ge \tr
  \mathcal{M}(b_1,\cdots,b_{2n}) \ge \tr (\mathcal{M}(2)^l).
\]
\end{proof}

Relating an arbitrary Anosov map to one with the form $\mathcal{M}(a_1,\cdots,a_{2n})$,
we can extend Proposition \ref{prop:StandardIneq} to Anosov elements.

\begin{COR}
	\label{cor:GeneralIneq}
	Let $f$ be an Anosov map, identified with a hyperbolic element in $\PSL(2,\mathbb{Z})$. Then
        \[
          \tr(f) \geq \tr( \mathcal{M}(2)^l),
      \]
      where $l$ is the translation length of $f$ on $\mathcal{F}$.
\end{COR}
\begin{proof}
  A slight change in the proof of Proposition \ref{prop:stdconj}
  actually shows that $f$ is conjugate to one associated with a standard
  ladder in which \textit{both $0$ and $\infty$ are pivot points.}
  Since conjugate Anosov maps have the same traces and translation lengths on
  $\mathcal{F}$, we may assume the associated ladder of $f$
  has $\overline{0,\infty}$ as its rung with $0, \infty$ being pivot
  points. As $\det f=1$, we have $\tr f=\tr(f^{-1})$. Also note that
  $f$ and $f^{-1}$ have the same translation length on $\mathcal{F}$.
  Since either one of $f$ or $f^{-1}$ will send $0$ to a positive
  rational number, it is possible to write either one as
  $\mathcal{M}(a_1,\cdots, a_{2n})$ for some positive integers
  $a_1,\cdots,a_{2n}$. Finally, by Proposition \ref{prop:StandardIneq}, we have
  \[
    \tr(f)=\tr(f^{-1}) = \tr \mathcal{M}(a_1,\cdots,a_{2n}) \geq \tr(\mathcal{M}(2)^l).
  \]
\end{proof}

Finally, we establish our effective bound of the ratio of Teichm{\" u}ller to curve graph translation lengths.

\begin{THM}
  \label{thm:boundsOfRatio}
  Let $f$ be an Anosov map. Then the ratio $\frac{l_{\mathcal{T}}(f)}{l_{\mathcal{C}}(f)}$ is
  \begin{enumerate}[(1)]
  \item bounded below $\log(1+\sqrt{2}) \simeq 0.8814$, and the
    minimum is achieved when $f=\mathcal{M}(2,2)^n$ for $n \ge 1$.
  \item unbounded above. More precisely, there exists a sequence of
    Anosov maps $\{f_n\}$ such that
    \[ \lim_{n \to \infty }\frac{l_{\mathcal{T}}(f_n)}{l_{\mathcal{C}}(f_n)} = \infty.
    \]
  \end{enumerate}
\end{THM}

\begin{proof}
  \begin{enumerate}[(1)]
		
  \item Note that for every positive integer $n$, we have
    \[
      l_{\mathcal{T}}(f^n)=nl_{\mathcal{T}}(f), \quad l_{\mathcal{C}}(f^n)=nl_{\mathcal{C}}(f), \quad \frac{l_{\mathcal{T}}(f^n)}{l_{\mathcal{C}}(f^n)}=\frac{l_{\mathcal{T}}(f)}{l_{\mathcal{C}}(f)}.
    \]
    Let $\lambda>1$ be the dilatation of $f$, which is identical to
    the largest eigenvalue of $f$. Thus
    $l_{\mathcal{T}}(f) = \log \lambda$, and
    $\tr f = \lambda + \frac{1}{\lambda}$. As both $x \mapsto \log x$ and $x \mapsto x+\frac{1}{x}$ are increasing functions for $x>1$,
    \[
      \tr(f) \geq \tr(g) \qquad \text{if and only if} \qquad
      l_{\mathcal{T}}(f) \geq l_{\mathcal{T}}(g).
    \]

    Now write $l_{\mathcal{C}}(f)=l$. Then
    $l_{\mathcal{C}}(f^2)=2l=l_{\mathcal{C}}(\mathcal{M}(2)^{2l})$ and
    $\tr f^2 \ge \tr (\mathcal{M}(2)^{2l})$ by Corollary \ref{cor:GeneralIneq}.
    Hence,
    \[
      \frac{l_{\mathcal{T}}(f)}{l_{\mathcal{C}}(f)}=\frac{l_{\mathcal{T}}(f^2)}{l_{\mathcal{C}}(f^2)}\geq
      \frac{l_{\mathcal{T}}(\mathcal{M}(2)^{2l})}{l_{\mathcal{C}}(\mathcal{M}(2)^{2l})}
      =\frac{l_{\mathcal{T}}(\mathcal{M}(2)^{2})}{l_{\mathcal{C}}(\mathcal{M}(2)^{2})}=\log(1+\sqrt{2}).
    \]
    Therefore, we established the lower bound of
    $\frac{l_{\mathcal{T}}(f)}{l_{\mathcal{C}}(f)}$ as
    $\log(1+\sqrt{2})$, which is realized by
    $\mathcal{M}(2,2)^n= \begin{pmatrix} 5 & 2 \\ 2 & 1 \end{pmatrix}^n$ for every positive integer $n$.
		
  \item For the upper bound, we have $f_n =
    \begin{pmatrix}
      n+1 & 1 \\ n & 1
    \end{pmatrix}$ whose dilatation falls in the interval $(n+1,n+2)$,
    for the trace is the sum of the dilatation and its reciprocal. Also,
    the translation length of $f_n$ is 1 because its invariant ladder
    is of type $(1,n)$. Thus we established a sequence of Anosov maps
    $\{f_n\}$ such that
    \[
      \frac{l_{\mathcal{T}}(f_n)}{l_{\mathcal{C}}(f_n)} > n+1,
    \]
    for each $n \ge 1$. Therefore,
    $\lim_{n \to
      \infty}\frac{l_{\mathcal{T}}(f_n)}{l_{\mathcal{C}}(f_n)} =
    \infty$, and so ratio is unbounded above.
  \end{enumerate}
\end{proof}

\subsection{Evenly Spread Length Spectrum of Translation Lengths}

Motivated from \cite[Theorem 9]{baik2017typical}, we now present how
typical is the set $m_f = \#\{[f'] \in \PSL(2,\mathbb{Z}):
  l_{\mathcal{C}}(f')=l_{\mathcal{C}}(f)\}$. Informally, the result can
be phrased as ``the spectrum of translation length on $\mathcal{F}$ is
\textit{evenly spread}.''

\begin{THM}
  \label{thm:typical}
  For any positive integer $k$, we have
  \begin{align}
    \label{eqn:typical}
    \lim_{R \to \infty} \frac{\#\{[f] \in cl(\PSL(2,\mathbb{Z})): \lambda(f)<R,\ \ m_f \ge k\}}{\#\{[f] \in cl(\PSL(2,\mathbb{Z})): \lambda(f)<R\}} \to 1,
  \end{align}
  where $cl(\PSL(2,\mathbb{Z}))$ is the set of conjugacy classes in
  $\PSL(2,\mathbb{Z})$, and $\lambda(f)$ is the dilatation of
  $f \in \PSL(2,\mathbb{Z})$.
\end{THM}

To prove this, we need the following asymptotic result on the number of conjugacy classes of $\SL(2,\mathbb{Z})$. For two functions $f,g : \mathbb{Z}^+ \to \mathbb{R}_{\ge 0}$, denote by $f \in w(g)$ if $\lim_{t \to \infty}\frac{f(t)}{g(t)}=\infty$.

\begin{LEM}[{\cite[Theorem 2]{chowla1980number}}]
  \label{lem:conjugacyclasses}
  Denote by $H(t)$ the number of conjugacy classes in $\SL(2,\mathbb{Z})$ with fixed trace $t$.
  For any $\theta > 0$, there exists $T(\theta)>0$ such that whenever $|t| > T(\theta)$, then $H(t) > |t|^{1-\theta}.$ In particular, $H \in w(\log).$  
\end{LEM}


\begin{proof}[Proof of Theorem \ref{thm:typical}]
  Define $N(R) = \#\{[f] \in cl(\PSL(2,\mathbb{Z})): \lambda(f) < R\}$,
  the number of conjugacy classes in $\PSL(2,\mathbb{Z})$ whose
  dilatation is less than $R$. Then by Lemma
  \ref{lem:conjugacyclasses},
  \begin{align*}
    N(R) &\ge \#\{[f] \in cl(\PSL(2,\mathbb{Z})): \tr(f) = \lceil R \rceil \} \\ &= \#\{[f] \in cl(\SL(2,\mathbb{Z})): \tr(f)=\lceil R \rceil \} \in w(\log R),
  \end{align*}
  where we used the fact that $\lambda(f)<R$ if and only if
  $\tr(f) \le \lceil R \rceil$.  Define
  $N_{i}(R) = \#\{[f] \in cl(\PSL(2,\mathbb{Z})): \lambda(f)<R,\ \
  l_{\mathcal{C}}(f) = i\}$, which is well-defined since conjugate
  elements have the same translation length. Then the numerator of
  ratio in \eqref{eqn:typical} can be rewritten as follows:
  \begin{align}
    \label{eqn:numerator}
    \#\{[f] \in cl(\PSL(2,\mathbb{Z})): \lambda(f)<R,\ \ m_{f} \ge k]\} = \sum_{i=1}^\infty N_{i}(R)\mathbf{1}_{[k,\infty)}(N_i(R)),
  \end{align}
  where $\mathbf{1}_A$ is the characteristic function on a set
  $A$. Note the infinite sum in \eqref{eqn:numerator} is, in fact, a
  finite sum, since the index $i$ is bounded above
  $c \log R := \log (1+\sqrt{2}) \log R$ by Theorem
  \ref{thm:boundsOfRatio}(1). Since
  $\sum_{i=1}^\infty N_i(R)= N(R) \in w(\log R)$, this suggests that
  for sufficiently large $R$, the RHS of \eqref{eqn:numerator}
  \[
    \sum_{i=1}^{\lfloor c\log R \rfloor}N_i(R) \mathbf{1}_{[k,\infty)}(N_i(R)),
  \]
  is minimized when all but one $\mathbf{1}_{[k,\infty]}(N_i(R))$
  vanish. More precisely, it is minimized when $N_j(R)=k-1$ except for
  one $j=i$, and $N_i(R)= N(R) - (\lfloor c\log R \rfloor -1) (k-1)$.
  Therefore, the ratio in \eqref{eqn:typical}:
    \begin{align*}
     1 \ge  \frac{\#\{[f] \in cl(\PSL(2,\mathbb{Z})): \lambda(f)<R,\ \ m_f \ge k\}}{\#\{f \in cl(\PSL(2,\mathbb{Z})): \lambda(f)<R\}} &\ge \frac{N(R) - (\lfloor c\log R \rfloor -1)(k-1)}{N(R)}\\
      &= 1 - (k-1)\frac{\lfloor c\log R \rfloor -1}{N(R)} \to 1
    \end{align*}
    as $R \to \infty$, since $N(R) \in w(\log R)$. Therefore, we established the limit \eqref{eqn:typical}.
\end{proof}

\clearpage
\appendix
\section{Algorithms for Ladders}
In this appendix, we provide two algorithms for ladders: to generate a
ladder bounded by two given edges(Algorithm \ref{alg:genLadder}), and to calibrate a ladder with odd length into one with even length(Algorithm \ref{alg:calLadder}).

\begin{algorithm}[h]
  \caption{Generate a ladder bounded by two edges in the Farey graph.
    \label{alg:genLadder}}
  \begin{algorithmic}[1]
    \Require{Two pairs of \textit{ExtRational}s $\{pq,rs\}$ is distinct from $\{xy,zw\}$ }
    \Function{generateLadder}{$pq,rs,xy,zw$}
      \Let{$pivotList$}{empty list}
      \Let{$typeList$}{empty list}
      \Let{$prevPivot$}{None} \Comment{stores the previous pivot point}
      \Let{$m$}{\textsc{FareySum}($pq,rs$)} \Comment{tells relative positions of $xy,zw$ to $pq,rs$}
      \While{True}
      \If{$m$ is in between $(pq,rs)$}
        \Let{$tu$}{\textsc{FareySum}($pq,rs$)}
      \Else
        \Let{$tu$}{\textsc{FareySubtract}($pq,rs$)}
      \EndIf
      \Let{$(d1,d2)$}{\textsc{FareySort}($pq,rs$)} \Comment{$d1>d2$.}
      \If{$m$ is in between $(d1,r)$}
      \Let{($choice$, $non-choice$)}{$d1, d2$}
      \Else
      \Let{($choice$, $non-choice$)}{$d2, d1$}
      \EndIf
      \If{$choice == prevPivot$}
      \State add the last element of $typeList$ by 1
      \Else
        \If{$prevPivot$ is None}
        \State append $nonchoice$ to the end of $pivotList$.
        \EndIf
      \State append $choice$ to the end of $pivotList$
      \State append 1 to the end of $typeList$
      \Let{$prev$}{$choice$}
      \EndIf
      \If{$\{choice, tu\} == \{xy,zw\}$} 
      \State append $tu$ to the end of $pivotList$
      \State {\bf return} $pivotList, typeList$
      \EndIf
      \EndWhile
    \EndFunction
  \end{algorithmic}
\end{algorithm}

\begin{algorithm}
  \caption{Calibrate an odd ladder into even ladder.
    \label{alg:calLadder}}
  \begin{algorithmic}
    \Require{$mn$ must be set as \textsc{FareySum}($firstPivot,secondPivot$)}
  \Function{calibrateLadder}{$pivotList,typeList,mn$}
    \If{the length of $pivotList$ is even} 
    \State {\bf return} $pivotList, typeList$ \Comment{no need to calibrate}
    \EndIf
    \Statex
    \Let{$intTranslate$}{the first element of $ladderType$}
    \State remove the first element of $ladderType$
    \Let{$lastElement$}{the last element of $pivotList$}
    \State remove the last element of $pivotlist$
    \Let{$lastPivot$}{the last element of $pivotList$}

    \For{$i=0$ to $intTranslate$}
    \If{the $mn$ is in between $lastPivot$ and $lastElement$}
      \Let{$lastElement$}{\textsc{FareySum}($lastPivot,lastElement$)}
    \Else
      \Let{$lastElement$}{\textsc{FareySubtract}($lastPivot,lastElement$)}
    \EndIf
    \EndFor
    \State remove the first element of $pivotList$
    \State append $lastElement$ to the end of $pivotList$
    \State add the last element of $typeList$ by $intTranslate$
    \State {\bf return} $pivotList, ladderType$
  \EndFunction
  \end{algorithmic}
\end{algorithm}

\clearpage 
\bibliographystyle{alpha}
\bibliography{mybib}

\end{document}